\documentclass[12pt, reqno]{amsart}
\usepackage{amsmath, amsfonts, amssymb, amscd, mathtools}
\usepackage[mathscr]{eucal}

\usepackage{enumerate}
\usepackage{enumitem}
\usepackage{graphicx}
\usepackage[all]{xy}
\usepackage[dvipsnames]{xcolor}
\usepackage{tikz-cd}
\usepackage[english]{babel}
\usepackage[utf8]{inputenc}
\usepackage[normalem]{ulem}
\usepackage{comment}
\usepackage{setspace}

\usepackage[
colorlinks=true,
citecolor=red,
urlcolor=blue,
linkcolor=blue,
bookmarksopen=true
]{hyperref}

\newtheorem{theorem}{Theorem}[section]
\newtheorem{proposition}[theorem]{Proposition}
\newtheorem{lemma}[theorem]{Lemma}
\newtheorem{remark}[theorem]{Remark}
\newtheorem{definition}[theorem]{Definition}
\newtheorem{example}[theorem]{Example}
\newtheorem{corollary}[theorem]{Corollary}

\newcommand{\thmref}[1]{Theorem~\ref{#1}}

\newcommand{\secref}[1]{Section~\ref{#1}}
\newcommand{\lemref}[1]{Lemma~\ref{#1}}
\newcommand{\propref}[1]{Proposition~\ref{#1}}

\newcommand{\corref}[1]{Corollary~\ref{#1}}
\newcommand{\remref}[1]{Remark~\ref{#1}}

\newcommand{\subsecref}[1]{Subsection~\ref{#1}}

\renewcommand{\hom}{\mathrm{Hom}}

\newcommand{\im}{\mathrm{Im}}

\newcommand{\fix}{\mathrm{Fix}}
\begin{document}
	\baselineskip=15.5pt
	\title[Rigidity of cohomology automorphisms and coincidence theory]{Rigidity of cohomology automorphisms of homogeneous spaces and coincidence theory}

	\author[M. Mandal]{Manas Mandal}
	\address{Indian Institute of Technology Kanpur, Kanpur 208016, India}
	\email{manas.imsc@gmail.com}
	
	\author[D. Setia]{Divya Setia}
	\address{Institute of Mathematics, Polish Academy of Sciences, Kraków 31-027, Poland}
	\email{divyasetia01@gmail.com}
	
	\subjclass[2020]{Primary: 55S37, 08A35,  55M20; Secondary: 14M15}
	\keywords{cohomology endomorphisms, homogeneous spaces, complex Grassmann manifolds, generalized Dold spaces,  fixed point theory, coincidence theory}

	\thispagestyle{empty}
	\date{}

	\begin{abstract}
    We obtain a rigidity phenomena of rational cohomology automorphisms of certain homogeneous spaces, in the presence of external cohomology classes arising from spaces with trivial cup product in rational cohomology algebra. We classify graded endomorphisms of the rational cohomology algebra of the product of a sphere and a complex Grassmannian, whose images are nonzero  in the second cohomology of the Grassmannian.
		We also derive necessary conditions for the generalized Dold spaces to satisfy the coincidence property, in particular the fixed-point property. As an application of our results, we obtain several sufficient conditions for the existence of a point of coincidence between a pair of continuous functions on certain generalized Dold spaces. 
	\end{abstract}

	\maketitle
	\begin{center}
	\end{center}
	
	\section{Introduction} \label{intro}
	The classification of endomorphisms of the rational cohomology algebra of formal spaces was greatly motivated by Sullivan's theory where it was  proved that rational homotopy class of self-maps are completely determined by the induced graded endomorphisms of their rational cohomology algebras. 
	
	In \cite{brewster}, the authors developed the foundational work by classifying automorphisms of the rational cohomology algebra of complex Grassmannian. Their results were generalized in \cite{hoffman}, where the author classified graded endomorphisms of the rational cohomology algebra of complex Grassmannian which are nonzero on dimension $2$. Further, he conjectured that every graded endomorphism  vanishing on dimension two is necessarily trivial. This conjecture was proved in \cite{glover-homer} for several cases.
	
	The cohomology endomorphisms are also studied for a variety of homogeneous spaces $G/H$, where $G$ is a compact connected Lie group and $H$ is a closed
	subgroup of maximal rank. This is a topic of interest since past fifty years and are studied in several papers \cite{shiga-tezuka2, brewster-homer, hoffman-homer, Papadima, duan, duan-fang, duan-zhao, lin, goswami-sarkar}. 
	
		We obtain a rigidity phenomena of the rational  cohomology automorphisms of the homogeneous spaces $G/H.$
		Consider a finite CW complex $X$ with trivial cup product in $H^*(X;\mathbb Q)$. Then, for any graded endomorphism $\phi$ of $H^*(X\times G/H;\mathbb Q)$ such that $i^*\circ \phi \circ p^*$ is an automorphism on $H^*(G/H;\mathbb Q)$, where  $i^*: H^*(X\times G/H;\mathbb Q) \twoheadrightarrow H^*(G/H;\mathbb Q)$ is the projection induced by the inclusion $i: G/H\hookrightarrow X\times G/H$ and $p^*: H^*(G/H;\mathbb Q) \hookrightarrow H^*(X\times G/H;\mathbb Q)$ is the inclusion induced by the projection $p: X \times G/H \twoheadrightarrow G/H$, we prove that $\phi \big(H^*(G/H;\mathbb Q)\big)=H^*(G/H;\mathbb Q)$ in the following result.
		\begin{theorem}\label{rigidity theorem in intro}
			Let $X$ be a path-connected finite CW complex with trivial cup product in its rational cohomology algebra.
			Let $G$ be a compact connected Lie group and $H$ be a closed subgroup of maximal rank.
			Consider a graded endomorphism $$\phi: H^*(X\times G/H;\mathbb Q)\to H^*(X\times G/H;\mathbb Q)$$ 
			such that $i^*\circ\phi\circ p^*$ is  an automorphism $$\phi_0: H^*(G/H;\mathbb Q)\to H^*(G/H;\mathbb Q).$$
			
			Then, $\phi(z)=\phi_0(z)$ for all $z\in H^*(G/H;\mathbb Q).$
	\end{theorem}
	
		Next, we focus our attention to a particular case where $X$ is a sphere $\mathbb S^m$ and $G/H$ is a complex Grassmannian $\mathbb CG_{n,k}$ (consisting of $k$-dimensional subspaces in $\mathbb C^n$)
		and we classify graded endomorphisms of the rational cohomology algebra $H^*(\mathbb{S}^m \times \mathbb{C}G_{n,k};\mathbb{Q})$, whose images are nonzero in $H^2(\mathbb CG_{n,k},\mathbb Q)$. Our main motivation for this special choice $\mathbb S^m\times \mathbb CG_{n,k}$ is to study the coincidence theory (in particular, fixed-point theory) of certain generalized Dold spaces.
	
	The rational cohomology algebra of the product
	\[
	H^*(\mathbb S^m \times \mathbb CG_{n,k};\mathbb Q)\cong H^*(\mathbb S^m,\mathbb Q)\otimes H^*(\mathbb CG_{n,k};\mathbb Q)
	\]
	is generated by $u,c_1,c_2,\dots,c_k$, where $H^*(\mathbb{C}G_{n,k};\mathbb{Q})$ (resp. $H^*(\mathbb S^m;\mathbb Q)$) is generated by certain Chern classes $c_1, \dots, c_k$ (resp. $u$). Using \thmref{rigidity theorem in intro}, we obtain the following result.
	\begin{theorem}
		Let $\phi$ be a graded endomorphism of 
		$H^*(\mathbb S^{m}\times\mathbb C G_{n,k};\mathbb Q)$ 
		satisfying $\phi(c_1)\neq a u, \, a\in\mathbb Q$.  
		Then there exists a nonzero rational $\lambda$ such that the following holds.
		\begin{enumerate}\label{result main}
			\item If $k < n - k,$ 
			$$ \phi(c_i) = \lambda^i c_i, \forall i \in \{1,2,\dots,k\}$$
			If $k = n - k$, there is an additional possibility of $\phi$ that it is induced by the homeomorphism 
			\[
			\mathbb{C}G_{2k,k} \longrightarrow \mathbb{C}G_{2k,k}, 
			\quad L \longmapsto L^{\perp},
			\]
			where $L^{\perp}$ denotes the orthogonal complement of the $k$-plane $L$ in $\mathbb{C}^{2k}$.
			
			\item 
			The image of $H^*(\mathbb{S}^m;\mathbb{Q})$ under $\phi$ lies in 
			$H^*(\mathbb{S}^m;\mathbb{Q})$ or in $H^*(\mathbb{C}G_{n,k};\mathbb{Q})$ i.e. $$\phi(u) = \mu u,\, \mu \in \mathbb{Q}, \text{ or } \phi(u) \in H^*(\mathbb{C}G_{n,k};\mathbb{Q}) \text{ with } (\phi(u))^2 =0.$$
			
		\end{enumerate}
	\end{theorem}
	Unlike the case of the complex Grassmannian, we cannot expect a graded
	endomorphism of $H^*(\mathbb S^m \times \mathbb CG_{n,k};\mathbb Q)$ to be
	trivial merely because it vanishes in $H^2(\mathbb{C}G_{n,k}; \mathbb{Q})$. 
	In fact, we prove the following result.
	\begin{theorem}
		For each $  i\in\{1,2,\dots, k\}$, choose $P_i \in H^{2i-m}({\mathbb C G_{n,k}};\mathbb{Q})$  and $Q \in H^*(\mathbb S^m;\mathbb Q)\cup H^*({\mathbb CG_{n,k}};\mathbb{Q})$ with $Q^2=0$  in $ H^*(\mathbb{S}^m\times \mathbb{C}G_{n,k};\mathbb{Q})$. Then there exist a graded endomorphism $\phi$ on $H^*(\mathbb{S}^m\times \mathbb{C}G_{n,k};\mathbb{Q})$ such that 
		\[
		\phi(c_i)=uP_i, \; \forall i\in I, \text{ and  } \quad \phi(u)=Q.
		\]
	\end{theorem}
	
	We observe that if a continuous function on $\mathbb S^m \times \mathbb CG_{n,k}$ stabilizes a copy of Grassmannian then the induced cohomology endomorphism stabilizes the subalgebra $H^*(\mathbb S^m ;\mathbb Q)$, in the following result.
	\begin{theorem}
		Let $f$ be a continuous map on $\mathbb S^m\times \mathbb CG_{n,k}$ such that $f(\{{x_0}\}\times \mathbb CG_{n,k})\subseteq \{{x_0}\}\times \mathbb CG_{n,k} $ for some $x_0\in \mathbb S^m.$ Then $f^*(u) = \mu u$ for some $\mu \in \mathbb Z$. 
	\end{theorem}
	
	Our study is also motivated by the theory of generalized Dold spaces as the product space $\mathbb S^m \times \mathbb CG_{n,k}$ is a double cover of certain  generalized Dold spaces (GDS). 
	The classical Dold manifolds were introduced in \cite{dold} to construct odd-dimensional generators for Thom's unoriented cobordism ring. In this paper, we are interested in the GDS introduced in \cite{nath-sankaran} and defined as  
	\[
	P(m,n,k):=\mathbb S^m\times \mathbb CG_{n,k}/\!\!\sim, \text { where } (s,L)\sim (-s,\bar L).
	\]
	As an application of \thmref{result main}, we describe endomorphisms of $H^*(P(m,n,k);\mathbb{Q})$ induced by continuous functions on $P(m,n,k)$. Using this description, we prove that every automorphism of $H^*(P(m,n,k);\mathbb{Q})$ induced by a continuous function on $P(m,n,k)$ lifts to an automorphism of $H^*(\mathbb S^m\times \mathbb CG_{n,k}; \mathbb{Q})$ if $n>2$. 
	
	Our broader aim is to apply \thmref{result main} to obtain results in coincidence theory. Coincidence theory has been extensively studied in \cite{hoffman-noncoin, glover-homer coin, wong}. 
	A pair $(X,g)$ where $g$ is a continuous function on $X$ is said to have the coincidence property if $g$ has a point of coincidence with every continuous function on $X$. In particular, if $g$ is the identity map, then the coincidence property is same as the fixed-point property of $X$.
	
	We generalize Theorem $2$ of \cite{glover-homer} to the setting of coincidence theory and prove the following result.
	\begin{theorem}
		Consider a complex Grassmannian $\mathbb{C}G_{n,k}$ with $k(n-k)$ is even and either $k \leq 3$ and $n > 2k$, or $k>3$ and $n>2k^2 -1$. Let $g $ be a continuous map on $\mathbb{C}G_{n,k}$ with nonzero Brouwer degree. Then the pair $(\mathbb{C}G_{n,k}, g)$ satisfies the coincidence property.
	\end{theorem}

	To conclude, using the Lefschetz coincidence theorem and \thmref{result main}, we obtain multiple situations when two continuous functions on the generalized Dold space $P(m,n,k)$ are guaranteed to have a point of coincidence, and prove the following result.
	\begin{theorem}\label{coin points of f and g}
		Let $P(m,n,k)$ be a generalized Dold manifold with $k<n-k$ and $k(n-k)$ even. Let $f$ and $g$ be two continuous maps on $P(m,n,k)$ and $\tilde f, \tilde g$ be their lifts on the double cover $\mathbb S^m\times \mathbb CG_{n,k}$ of $P(m,n,k)$ such that
		\begin{enumerate}	
			\item $g^*$ is an automorphism of $H^*(P(m,n,k);\mathbb Q)$. 
			\item $\tilde{f}^*(c_1) \neq au,\, a \in \mathbb{Q}$.
			\item $\deg(p\circ g \circ s)\neq -\deg (p\circ f\circ s)$ if $m$ is odd.
		\end{enumerate}
		where $s$ denotes a section of the $\mathbb CG_{n,k}$-bundle projection $p: P(m,n,k)\to \mathbb RP^m$. Then, there is a point of coincidence of $f$ and  $g$.
	\end{theorem}
	As an application of \thmref{coin points of f and g}, we obtain certain pairs $(P(m,n,k),g)$ that satisfy the coincidence property.
	
	The paper is organized as follows: \\
	In \secref{section 2}, we develop the necessary background and recall some relevant results. \secref{section 3} is devoted to the study of rigidity of cohomology automorphisms of homogeneous spaces and graded endomorphisms of the rational cohomology algebra of  $\mathbb S^{m}\times \mathbb C G_{n,k}$, from which we extract several consequences. These are applied in \secref{section 4} to obtain the coincidence-theoretic results.

	\section{Preliminaries}  \label{section 2}
	
	In this section, we discuss some preliminaries and recall some results that will be required to proceed with our study.

	\subsection{Cohomology of complex Grassmannians}

	Let $\mathbb{C}G_{n,k}$ denote the complex Grassmannian consisting of complex $k$-planes in $\mathbb{C}^n$. 
	Let $\gamma_{n,k}$ and $\beta_{n,k}$ denote the canonical complex $k$-plane and $(n-k)$-plane bundles, respectively, over $\mathbb{C}G_{n,k}$.
	Let the total Chern classes of the vector bundles $\gamma_{n,k}$ and $\beta_{n,k}$ be denoted by $c(\gamma_{n,k}) = c$ and $c(\beta_{n,k}) = \bar{c}$, respectively. Thus,
	$$c = 1 + c_1 + c_2 + \cdots + c_k, \quad \bar{c} = 1 + \bar{c}_1 + \bar{c}_2 + \cdots + \bar{c}_{n-k},$$
	where $c_i$ and $\bar{c}_i$ denote the $i$-th Chern classes of $\gamma_{n,k}$ and $\beta_{n,k}$, respectively.
	Since $\gamma_{n,k} \oplus \beta_{n,k} \cong \varepsilon_{\mathbb{C}}^n$, it follows that  $c \cdot \bar{c} = 1$.
	The  cohomology ring of the complex Grassmannian is well known and given by  
	$$
	H^*_{\mathbb{C}G}:=H^*(\mathbb{C}G_{n,k};\mathbb{Q}) \cong \mathbb{Q}[c_1, c_2, \dots, c_k, \bar{c}_1, \bar{c}_2, \dots, \bar{c}_{n-k}]/\langle h_r: 1\leq r\leq n\rangle,$$  
	where  the relations $h_r$ for $r = 1, 2, \dots, n$ are induced from the homogeneous parts of the equation $c\cdot \bar c=1$ and given by  
	\[
	h_r := \sum_{i+j=r} c_i \bar{c}_j.
	\]  
	Using the relations $h_r, r=1,2,...,n-k$, the generators $\bar{c}_i$ for $i = 1, 2, \dots, n-k$ can be expressed inductively in terms of $c_i$ for $i = 1, 2, \dots, k$.  Consequently, the relations $h_r$ for $r = n-k+1, \dots, n$ become homogeneous polynomials in $c_i$ of degree $2r$, where the degree of each $c_i$ is $2i$. Then the cohomology ring $H^*_{\mathbb{C}G}$ can be rewritten as  
	\begin{equation}\label{cohomo of grass}
		\mathbb{Q}[c_1, c_2, \dots, c_k]/\langle h_{n-k+1}, h_{n-k+2}, \dots, h_n \rangle.
	\end{equation}
	Since there are no relations among the generators $c_i$ for $i = 1, 2, \dots, k$ up to degree $2(n-k)$, the set of all monomials of degree $2r$ in terms of $c_1, c_2, \ldots, c_k$ forms a $\mathbb{Q}$-basis of $H^{2r}(\mathbb{C}G_{n,k};\mathbb{Q})$ for $r \leq n-k$.
	
	From now on, we denote the indexing set $\{1,2,\dots, k\}$ by $I$.
	\begin{remark}
		We can assume  $k\leq n-k$ for  $\mathbb{C}G_{n,k}$ as $\mathbb{C}G_{n,k}$ is homeomorphic to $\mathbb{C}G_{n,n-k}$ by using orthogonal complementation.
	\end{remark}
	
	\subsection{Cohomology of homogeneous spaces} 
		Let $G$ be a compact connected Lie group and $H$ be a closed subgroup of maximal rank. Denote by $H^*(G/H;\mathbb Q)$, the cohomology algebra of the homogeneous space $G/H$. Now we recall a result given in \cite{shiga-tezuka}.  
		\begin{theorem}[\cite{shiga-tezuka}, Theorem \(A^{'}\)]\label{Tezuka}
			Let $D_i(H^*(G/H;\mathbb{Q}))$ be the $\mathbb{Q}$-vector space of $\mathbb{Q}$-derivations of $H^*(G/H;\mathbb{Q})$ which decreases the degree by $i>0$ where $G$ is a connected, compact Lie group and $H$ is a closed subgroup of maximal rank.
			Then, for all $i$, $$D_i(H^*(G/H;\mathbb{Q})) =0.$$  
		\end{theorem}
		The complex Grassmannian $\mathbb{C}G_{n,k}$ is a homogeneous space and can be represented as the quotient of the unitary group $U(n)$ by the stabilizer subgroup $U(k)\times U(n-k)$ that is
		\begin{equation}\label{cgn as hom}
			\mathbb{C}G_{n,k} = U(n)/ (U(k)\times U(n-k)). 
		\end{equation} 
	Consider a topological space $X$ whose cohomology has trivial cup product in positive degrees. Using the cohomology algebras $H^*(G/H;\mathbb Q)$ and $H^*(X;\mathbb Q)$ , one can study the cohomology algebra $H^*(X\times G/H;\mathbb Q)$. Indeed, the cohomology algebras $H^*(G/H;\mathbb Q)$ and $H^*(X;\mathbb Q)$ can be viewed as subalgebras of $H^*(X\times G/H;\mathbb Q)$ via the images of induced monomorphisms in cohomology by the first and second projection maps, respectively. Let $i: G/H \hookrightarrow X \times G/H$ be an inclusion of second factor and $p: X \times G/H \twoheadrightarrow G/H$ the projection onto the second factor. This induces the following maps 
	\begin{equation}\label{epi and mono}
		i^*: H^*(X\times G/H;\mathbb Q) \twoheadrightarrow H^*(G/H;\mathbb Q), \quad p^* : H^*(G/H;\mathbb Q) \hookrightarrow H^*(X\times G/H;\mathbb Q).
	\end{equation}
	

	\subsection{Graded endomorphisms on $\mathbf{H^*_{\mathbb{C}G}}$}
	It was conjectured in \cite{O} that any graded endomorphism $\phi$ of the cohomology algebra $H^*_{\mathbb{C}G}$ is an\textit{ Adams }map when $k < n - k$; that is, there exists a rational $\lambda$ such that
	$\phi(c_i) = \lambda^i c_i$, for all $ i \in I.$ Glover and Homer (see \cite{glover-homer}) and Hoffman (see \cite{hoffman}) proved the conjecture under the following hypothesis respectively:
	\begin{align}
		\text{Either } k \leq 3 \text{ and } n > 2k \text{, or } k>3 \text{ and } n>2k^2 -1. \label{Homer}\\
		\text{ The graded endomorphism } \varphi  \text{ of } H^*_{\mathbb{C}G} \text{ satisfies } \varphi(c_1) = \lambda c_1, \lambda\neq 0.  \label{Hoff}
	\end{align}
	Let us recall those results proved in \cite{glover-homer, hoffman} that will be used in the rest of this paper.  
	\begin{theorem}[\cite{glover-homer}, Theorem 1, \cite{hoffman}, Theorem 1.1]\label{hom and hof}
		(i) Assume that the hypothesis \eqref{Homer} is satisfied. Then for every graded endomorphism $\varphi$ on $ H^*(\mathbb{C}G_{n,k}; \mathbb{Q})$, there exists a rational $\lambda$ such that
		\[\varphi(c_i) = \lambda^i c_i,  \quad \forall i \in I.\]
		(ii) Assume that the hypothesis \eqref{Hoff} is satisfied. Then, we have
		$$\varphi(c_i) = \begin{cases}
			\lambda^i c_i,   \forall i \in I& \text{ if } k<n-k,\\
			\lambda^i c_i,  \forall i \in I \quad \text{ or } \quad(-\lambda)^i (c^{-1})_i,    \forall i \in I & \text{ if } k= n-k,
		\end{cases}$$
		where $ (c^{-1})_i $ is the $ 2i $-dimensional part of the inverse of $ c = 1 + c_1 + \cdots + c_k $ in $ H^*(\mathbb CG_{n,k}; \mathbb{Q}) $.
	\end{theorem}

	\subsection{Generalized Dold spaces}  \label{gen dold}
	In \cite{dold}, the author introduced the notion of \textit{classical Dold manifolds}  
	$P(m,n) := \mathbb{S}^m \times \mathbb{C}P^n / \!\! \sim$  
	where $ (s, L) \sim (-s, \bar{L}) $, where the involution $L\mapsto\bar L$ on $\mathbb CG_{n,k}$ is induced from the standard conjugation on $\mathbb C^n$, to construct generators in odd dimensions for René Thom's unoriented cobordism ring. 

	In \cite{nath-sankaran, mandal-sankaran}, the authors generalized the notion of classical Dold manifolds by replacing the sphere $ \mathbb{S}^m $ with an arbitrary topological space $ S $ equipped with a free involution $ \alpha $, analogous to the antipodal map on $ \mathbb{S}^m $, and $ \mathbb CP^n $ with an arbitrary topological space $ X $ with an involution $ \sigma: X \to X $ having a nonempty fixed-point set, analogously to complex conjugation on $\mathbb CP^n$. Then the quotient space  
	\begin{equation}\label{gen dold space}
		P(S, \alpha, X, \sigma) := S \times X / \!\! \sim, \quad \text{where } (s, x) \sim (\alpha(s), \sigma(x)), 
	\end{equation}  
	is called \textit{generalized Dold space} (in short GDS), often denoted simply as $ P(S, X) $. Moreover, the quotient map  
	$ S \times X \to P(S,X) $  
	is a double covering map. 
	
	Let us fix a notation $ Y $ for $ S/\!\!\sim_{\alpha} $, where $ s\sim_{\alpha} \alpha(s), \forall s\in S $. Then, a GDS $ P(S,X) $ is the total space of a fiber bundle  
	$X\hookrightarrow P(S,X) \twoheadrightarrow Y $, where the fiber bundle projection is \begin{equation}\label{proj}
		p: P(S,X) \twoheadrightarrow Y, \quad [s,x]\mapsto [s].
	\end{equation}  Choosing a fixed-point of $\sigma$, say $x_0\in \text{Fix}(\sigma)\neq \emptyset$, we can construct a section of the fiber bundle \begin{equation}\label{sectio}
		s: Y \hookrightarrow P(S,X), \quad [s]\mapsto [s,x_0].
	\end{equation}    
	In fact, we have an embedding  
	$Y\times \text{Fix}(\sigma)\hookrightarrow P(S,X),$  
	where $ \text{Fix}(\sigma) \subseteq X $ has the subspace topology induced from $ X $.

	\subsection{Rational cohomology of $\mathbf{P(\mathbb S^m,\mathbb CG_{n,k})}$} \label{gds}

	The GDS $P(\mathbb S^m,\mathbb CG_{n,k})$ is defined as
	\[
	\mathbb S^m\times \mathbb CG_{n,k}/\!\!\sim, \text { where } (s,L)\sim (-s,\bar L),
	\]
	for which, $\mathbb S^m$ is equipped with the free action generated by the antipodal map $\alpha$ and the involution $\sigma: L \mapsto \bar L$ on $\mathbb CG_{n,k}$ is induced from the standard complex conjugation on $\mathbb C^n.$ We denote $P(\mathbb S^m,\mathbb CG_{n,k})$ simply by $P(m,n,k).$
	By the K\"unneth formula, we have
	\begin{equation}\label{Cohomology of H_times}
		H_{\times}^* := H^*(\mathbb{S}^m \times \mathbb{C}G_{n,k}; \mathbb{Q}) \cong H^*(\mathbb{S}^m; \mathbb{Q}) \otimes H^*(\mathbb{C}G_{n,k}; \mathbb{Q}) \cong \frac{\mathbb{Q}[u, c_1, \dots, c_k]}{\langle u^2, h_{n-k+1}, \dots, h_n \rangle}
	\end{equation}
	where $u \in H^m(\mathbb S^m; \mathbb Q)$ denotes the generator corresponding to the fundamental class of $\mathbb S^m$. Note that \begin{equation}\label{H in terms of u}
		H^*_{\times} \cong H^*_{\mathbb{C}G}[u]/\langle u^2 \rangle \cong H^*_{\mathbb{C}G} \oplus u H^*_{\mathbb{C}G},
	\end{equation} where the latter isomorphism is a $\mathbb Q$-module isomorphism. We have that $H^*_{\mathbb{C}G}$ is a subring of $H^*_\times$.
	The product involution $\theta:= \alpha\times\sigma$ on $\mathbb S^m\times \mathbb CG_{n,k}$ induces an involution $\theta^*$ on $H^*_\times$ given by
	\begin{equation}\label{defn of theta*}
		\theta^*(c_i) = (-1)^i c_i,i\in I, \quad \theta^*(u) =  
		\begin{cases}  
			u, & \text{if } m \text{ is odd}, \\  
			-u, & \text{if } m \text{ is even}.  
		\end{cases}  
	\end{equation}
	The cohomology ring $ H^*(P(m, n,k);\mathbb Q) $ was computed in \cite{mandal-sankaran2} and the following result was proved.
	\begin{theorem}[{\cite[Theorem 3.13]{mandal-sankaran2}}]\label{cohomology of P(m,n,k)}
		The cohomology algebra \( H^*(P(m,n,k); \mathbb{Q}) \) is isomorphic to the subalgebra
		$\mathrm{Fix}(\theta^*) \subseteq H^*(\mathbb{S}^m \times \mathbb{C}G_{n,k}; \mathbb{Q}),$
		generated by the following elements:
		\begin{align*}
			u \, c_{2p-1},\quad c_{2j},\quad c_{2p-1} \, c_{2q-1},\; \forall 2p-1, 2q-1, 2j \in I, \text{ if } m \text{ is even};\\
			u,\quad c_{2j},\quad c_{2p-1} \, c_{2q-1},\; \forall 2p-1, 2q-1, 2j \in I, \text{ if } m \text{ is odd}.
		\end{align*}
	\end{theorem}
	A description of the cohomology algebra $ H^*(P(m,n,k); \mathbb{Q}) $, as a quotient of a polynomial algebra, can be deduced as a particular case in Theorem 3.14 of \cite{mandal-sankaran2}.

	
	\section{Rigidity of Cohomology Automorphisms}\label{section 3}
	
	 In this section, we  study an extended rigidity phenomena of rational cohomology automorphisms of certain homogeneous spaces in the presence of external cohomology classes. 
	\subsection{}	The following theorem is one of the main result of this section.
\begin{theorem}\label{rigidity theorem}
			
			Let $X$ be a path-connected finite CW complex with trivial cup product in its rational cohomology algebra.
			Let $G$ be a compact connected Lie group and $H$ be a closed subgroup of maximal rank.
			Consider a graded endomorphism $$\phi: H^*(X\times G/H;\mathbb Q)\to H^*(X\times G/H;\mathbb Q)$$ 
			such that $i^*\circ\phi\circ p^*$ is  an automorphism $$\phi_0: H^*(G/H;\mathbb Q)\to H^*(G/H;\mathbb Q)$$ where $i^*$ and $p^*$ are given in \eqref{epi and mono}.
			
			Then, $\phi(z)=\phi_0(z)$ for all $z\in H^*(G/H;\mathbb Q).$
			
		\end{theorem}
		\begin{proof}
			
			We are given that $X$ is finite CW complex which implies that $H^*(X;\mathbb Q)$ is finite dimensional $\mathbb Q$-algebra. 
			Choose a $\mathbb Q$-basis of $H^*(X;\mathbb Q)$ 
			$$B:=\{u_0,u_1,u_2,\ldots, u_r\}$$ 
			where $u_0=1\in H^0(X;\mathbb Q)\cong \mathbb Q$ and $\deg(u_i)>0$ for each $i=1,2,\ldots, r$. 
			

			
			By the K\"unneth formula, we have $H^*(X\times G/H;\mathbb Q)\cong H^*(X;\mathbb Q)\otimes H^*(G/H;\mathbb Q)$ as graded $\mathbb Q$-algebras. Hence, we have
			\begin{equation}\label{Kunneth}
				H^*(X\times G/H;\mathbb Q)\cong \bigoplus_{i=0}^r u_iH^*(G/H;\mathbb Q),
			\end{equation} where the cohomology algebra $H^*(G/H;\mathbb Q)$ is identified with the subalgebra $u_0 H^*(G/H;\mathbb Q)$.

			Let $p_0: \bigoplus_{i=0}^r u_iH^*(G/H;\mathbb Q)\twoheadrightarrow H^*(G/H;\mathbb Q)$ be the projection onto the zeroth summand and $i_0: H^*(G/H;\mathbb Q)) \hookrightarrow \bigoplus_{i=0}^r u_iH^*(G/H;\mathbb Q)$ be the inclusion into the zeroth summand. In fact, $p_0=i^*$ and $i_0=p^*$ as given in \eqref{epi and mono}.  We have the following commutative diagram.
			\begin{equation}\label{commutative diagram}
				\begin{tikzcd}
					{\bigoplus_{i=0}^r u_iH^*(G/H;\mathbb Q)} & {\bigoplus_{i=0}^r u_iH^*(G/H;\mathbb Q)} \\
					{H^*(G/H;\mathbb Q)} & {H^*(G/H;\mathbb Q)}
					\arrow["\phi", from=1-1, to=1-2]
					\arrow["{i^*}", two heads, from=1-2, to=2-2]
					\arrow["{p^*}", hook, from=2-1, to=1-1]
					\arrow["{\phi_0}", from=2-1, to=2-2]
				\end{tikzcd}
			\end{equation}
			It is given that the  composition $\phi_0:= i^* \circ \phi \circ p^*$ is an algebra automorphism of $H^*(G/H;\mathbb Q)$.
			Thus, for each  $z \in H^{*}(G/H;\mathbb Q) \subseteq H^*(X\times G/H;\mathbb Q) $, $\phi(z)$ can be written as
			\begin{equation}\label{phi x as sum}
				\phi(z) = \sum _{i=0}^r u_iP_{z,i}=\phi_0(z)+  \sum_{i=1}^r u_iP_{z,i},
			\end{equation} where $P_{z,i} \in H^*(G/H;\mathbb Q),i=0,1,\ldots,r$, and $P_{z,0}=\phi_0(z)$. 
			

			To prove that $\phi(z) = \phi_0(z)$, it is sufficient to show that $P_{z,i} =0,\, \forall i=1,2,\ldots, r$, in \eqref{phi x as sum}. Since $\phi_0$ is an automorphism, using $\phi_0^{-1}$ and \eqref{phi x as sum}, for each $i=1,2,\ldots, r$, we define $D_i: H^*(G/H;\mathbb Q)\to H^*(G/H;\mathbb Q)$ by
			$$D_i(z) = P_{\phi_0^{-1}(z),i},\, \forall z \in H^*(G/H;\mathbb Q).$$
			Equivalently, we have $D_i(\phi_0(z)) = P_{z,i}$.
			It is easy to check that each $D_i$ is a $\mathbb Q$-linear on $H^*(G/H;\mathbb Q)$.
			We shall prove that, for all $i>0,$ $D_i$ satisfies the Leibniz rule and hence defines a derivation on $H^*(G/H;\mathbb Q)$.
			
			Using \eqref{phi x as sum}, for $z,w\in H^*(G/H;\mathbb Q)$ we have \[
			\phi(z)=\phi_0(z)+ \sum _{i=1}^r u_iP_{z,i}\quad \text{ and } \quad
			\phi(w)=\phi_0(w)+\sum _{i=1}^r u_iP_{w,i}.
			\]
			We are given that $H^*(X;\mathbb Q)$ has a trivial cup product. Therefore, $u_iu_j=0, \, \forall i,j\ge1$. Thus, we have
			\begin{equation}\label{phi x phi y}
				\phi(z)\phi(w)= \phi_0(z)\phi_0(w)+\sum_{i=1}^ru_i\big(\phi_0(z)P_{w,i}+P_{z,i}\phi_0(w)\big). 
			\end{equation}
			Consequently, when $\phi(z)\phi(w)\in H^*(X\times G/H;\mathbb Q)$ is expressed with respect to the $H^*(X;\mathbb Q)$-basis $B$, the coefficient of $u_i$ is given by
			\[
			P_{zw,i}=\phi_0(z)P_{w,i}+P_{z,i}\phi_0(w).
			\]
			Equivalently, it can be written as
			\[
			D_i(\phi_0(zw))=D_i\big(\phi_0(z)\phi_0(w)\big)
			=\phi_0(z)\,D_i(\phi_0(w))+D_i(\phi_0(z))\,\phi_0(w).
			\] Therefore, $D_i$ satisfies the Leibniz rule.
			This proves that for each $i=1,2,\ldots,r$, $D_i$ is a derivation on $H^*(G/H;\mathbb Q)$. For $z \in H^j(G/H;\mathbb Q)$, we have $D_i(z) \in H^{j-\deg(u_i)}(G/H;\mathbb Q)$ which implies that the derivation $D_i$ decreases the degree by $\deg(u_i)>0$.  Using \thmref{Tezuka}, we conclude that $D_i$ is a zero derivation for all $i>0$. 
			
			Thus, for any $z\in H^*(G/H;\mathbb Q)$, we have $P_{z,i}= D_i(\phi_0(z))=0,\, \forall i>0$.
		\end{proof}

		
	\begin{remark}
		We provide a few examples of spaces $X$ that are considered in the Theorem~\ref{rigidity theorem}, i.e., the spaces $X$ with trivial cup product in rational cohomology.
		\begin{enumerate}
			\item Spheres $\mathbb{S}^m,\, m\geq0$.
			\item For an abelian group $G$ and $n\in \mathbb N$, the Moore space $M(G,n)$ is a topological space which has non-trivial reduced homology group $G$ only in dimension $n$ and assumed to be  simply connected if $n>1$. 
			\item The suspension $\Sigma Y$ of any finite CW complex $Y$. 
			\item The wedges of spaces discussed above provide further examples.
		\end{enumerate}
		\end{remark}	
		
			
		

	\subsection{Graded endomorphism of $H^*(\mathbb S^m\times \mathbb CG_{n,k};\mathbb Q)$}
		From now on, we restrict our focus to a  special case where $$X=\mathbb S^m, \quad G/H=\mathbb CG_{n,k}.$$ This set up is essential in our study to develop the coincidence theory in \secref{section 4} of certain spaces, called generalized Dold spaces defined in \subsecref{gds}.
	
	Now, we classify graded endomorphisms of the rational cohomology algebra $ H^*(\mathbb S^m\times \mathbb CG_{n,k}; \mathbb{Q}) $ whose images are nonzero in $H^2(\mathbb CG_{n,k};\mathbb Q)$. Our approach relies on the study of graded endomorphisms of $ H^*(\mathbb{C}G_{n,k}; \mathbb{Q}) $ from \cite{glover-homer} and \cite{hoffman}. Assume $m>0$ for the rest of this paper.
	
	The cohomology ring of the complex Grassmannian $ \mathbb{C}G_{n,k} $ is generated by the Chern classes $c_i,\forall i \in I $ as given in \eqref{cohomo of grass}. In \eqref{Cohomology of H_times}, we see that the cohomology ring of $S^m \times \mathbb{C}G_{n,k}$ is generated by $u, c_i, \forall i \in I$. Therefore, it is sufficient to describe the images of the generators to classify graded endomorphisms of $H_{\times}^*$. Hence, we have the following proposition.

	\begin{proposition}\label{main thm}
		Let $\phi$ be a graded endomorphism of $H^*_{\times}$ satisfying $\phi(c_1) \neq \mu u,\, \mu \in \mathbb{Q}$.  
		Then the following holds, \begin{enumerate}
			\item Either $\phi(u)=au$ for some $a \in \mathbb{Q}$, or $\phi(u) \in H^*_{\mathbb{C}G} \subseteq H^*_{\times}$ with $\phi(u)^2=0$ in $H^*_{\times}$.
			\item There exists $\lambda \in \mathbb Q\backslash\{0\}$ such that
			$$\phi(c_i) = \begin{cases}
				\lambda^i c_i,   \forall i \in I& \text{ if } k<n-k,\\
				\lambda^i c_i,  \forall i \in I \quad \text{ or } \quad(-\lambda)^i (c^{-1})_i,    \forall i \in I & \text{ if } k= n-k,
			\end{cases}$$
		\end{enumerate} where $ (c^{-1})_i $ is the $ 2i $-dimensional part of the inverse of $ c = 1 + c_1 + \cdots + c_k $ in $ H^*_{ \mathbb CG} $.
	\end{proposition}
	
	
	\begin{proof}

		From equation \eqref{Cohomology of H_times} and \eqref{H in terms of u}, we have $H^*_{\times}\cong \mathcal R/\mathcal I \cong H^*_{\mathbb{C}G} \oplus u H^*_{\mathbb{C}G}$, where $\mathcal R:=\mathbb Q[u,c_1,\ldots,c_k]$ and $\mathcal I:=\langle u^2, h_{n-k+1},\ldots,h_n\rangle$.
		

		
		Let $p_1: H^*_{\times} = H^*_{\mathbb{C}G} \oplus u H^*_{\mathbb{C}G} \to H^*_{\mathbb{C}G}$ be the projection onto the first summand and $i_1: H^*_{\mathbb{C}G} \hookrightarrow H^*_{\mathbb{C}G} \oplus u H^*_{\mathbb{C}G}$ be the inclusion into the first summand. The composite $\phi_1:= p_1 \circ \phi \circ i_1$ is a degree-preserving endomorphism of $H^*_{\mathbb{C}G}$. We have the following diagram:
		\begin{equation}\label{comm diagram}
			\begin{tikzcd}
				{H^*_{\mathbb CG}\oplus u H^*_{\mathbb CG}} & {H^*_{\mathbb CG} \oplus u H^*_{\mathbb CG}} \\
				{H^*_{\mathbb CG}} & {H^*_{\mathbb CG}}
				\arrow["\phi", from=1-1, to=1-2]
				\arrow["{p_1}", two heads, from=1-2, to=2-2]
				\arrow["{i_1}", hook, from=2-1, to=1-1]
				\arrow["{\phi_1}", from=2-1, to=2-2]
			\end{tikzcd}
		\end{equation}
		Thus, for  $x \in H^*_{\mathbb{C}G} \subset H^*_\times$, one can write
		$\phi(x) = \phi_1(x) + u P_x$
		for some $P_x \in H^*_{\mathbb{C}G} \subset H^*_\times$ because the kernal of $p_1$, $\ker(p_1) = u H^*_{\mathbb{C}G}$.
		This implies that \begin{equation}\label{defn of phi}
			\phi(c_i) = \phi_1(c_i) + u P_{c_i},\, \forall i \in I.
		\end{equation} For simplicity, denote $P_{c_i}$ by $P_i\in H^{2i-m}_{\mathbb CG}$ which is a polynomial in $c_1, \dots, c_k$ of degree $2i - m$ as $\deg c_i = 2i$ and $\deg u = m$.

		Since $\phi(c_1)\neq \mu u, \, \mu \in \mathbb{Q}$, that implies $\phi(c_1)$ is of the form $\lambda c_1+\mu u,\, \lambda,\mu \in \mathbb{Q}, \, \lambda \neq 0$. Then we have $\phi_1(c_1) = \lambda c_1,\, \lambda\neq 0$ on $H^*_{\mathbb{C}G}$.  By \thmref{hom and hof} \textit{(ii)}, we have 
		\begin{equation}\label{phi_1}
			\phi_1(c_i) = \begin{cases}
				\lambda^i c_i,   \forall i \in I& \text{ if } k<n-k,\\
				\lambda^i c_i,  \forall i \in I \quad \text{ or } \quad(-\lambda)^i (c^{-1})_i,    \forall i \in I & \text{ if } k= n-k,
			\end{cases}
		\end{equation} where $ (c^{-1})_i $ is the $ 2i $-dimensional part of the inverse of $ c = 1 + c_1 + \cdots + c_k $ in $ H^*_{ \mathbb CG} $. Now we prove both parts of the statement.

		\textit{proof of part (1):} Since $\phi$ is a graded endomorphism on $H^*_{\times}$, therefore $$\phi(u) = a u + P, \, a \in \mathbb{Q}, \text{ satisfying } (a u + P)^2 =0,$$ where $P$ is a homogeneous polynomial in $c_1, \dots, c_k$ of degree $m$. We have $P^2 + 2 a u P =0$ in $H^*_{\times}$. Using \eqref{H in terms of u}, we get that $2aP =0$ in $H^*_{\times}=\mathcal{R}/\mathcal{I}$. Hence, either $a=0$ or $P\in \mathcal{I}$.
		
			\textit{proof of part (2):} Let $i: \mathbb CG_{n,k}\hookrightarrow \mathbb S^m\times \mathbb CG_{n,k}$ be an inclusion of the second factor and $p: \mathbb S^m\times \mathbb CG_{n,k}\twoheadrightarrow \mathbb CG_{n,k}$ be the projection onto the second factor. Here, the induced maps in rational cohomology satisfies $i^*=p_1$ and $p^*=i_1.$ Using \eqref{comm diagram}, we have $i^*\circ \phi \circ p^*=\phi_1$. Since $\lambda\neq 0$ in \eqref{phi_1}, it follows that $\phi_1$ is a graded automorphism of $H^*_{\mathbb CG}.$ Using Theorem~\ref{rigidity theorem}, we have the proof.	\end{proof}

	\begin{remark}
		The condition $\phi(c_1)\neq \mu u,\, \mu \in \mathbb Q$ in		Proposition~\ref{main thm} classifies all graded endomorphisms $\phi$ of $H^*_\times$ whose image is nonzero in $H^2_{\mathbb CG}$ if $n>2$. In fact, $n>2$ implies $c_1^2\neq 0$ and $\phi(u) \neq ac_1,\, a \in \mathbb{Q}\setminus\{0\}$ as $\phi(u)^2=0$. Therefore, the only remaining possibility of graded endomorphisms whose image is non-zero in $H^2_{\mathbb{C}G}$ is $\phi(c_1)\neq \mu u,\, \mu \in \mathbb Q.$ 
		
		On the other hand, when $n=2,$ $\mathbb CG_{n,k}$ is either a point or $\mathbb S^2$ and the classification of graded endomorphisms of $H^*_\times$ is easy.
	\end{remark}

	\subsection{} In Proposition~\ref{main thm}, we assume that $\phi(c_1) \neq \mu u$. Let us try to look at the other case where $\phi(c_1) = \mu u$. To address this, we use part (i) of \thmref{hom and hof} which leads to the following proposition.

	\begin{proposition}\label{main thm 2}
		Assume that hypothesis \eqref{Homer} is satisfied.
		Let $\phi$ be a graded endomorphism such that $\phi(c_1)=\mu u,\, \mu \in \mathbb{Q}$ in $H^*_{\times}$. Then
		\begin{enumerate}
			\item Either $\phi(u)=a u$ for some $a \in \mathbb{Q}$, or $\phi(u) \in H^*_{\mathbb{C}G} \subseteq H^*_{\times}$ with $\phi(u)^2=0$ in $H^*_{\times}$.
			\item  $\phi(c_i) = uP_i, \, \forall i >1,$ where $P_i \in H^{2i-m}_{\mathbb CG}\subseteq H^*_\times$. 
		\end{enumerate}
	\end{proposition}
	\begin{proof} \textit{(1):} The proof of part \textit{(1)} is exactly the same as the proof of part \textit{(1)} of Proposition~\ref{main thm}. Therefore, we omit the details.
		
		\textit{(2):} Using \eqref{comm diagram}, we have that the map $\phi_1$ is a graded endomorphism on $H^*_{\mathbb{C}G}$ such that $\phi_1(c_1) =0$. By \thmref{hom and hof}, $\phi_1(c_i) =0, \, \forall i\in I$, then by \eqref{defn of phi}, we get $\phi(c_i) = uP_i$ for some $P_i \in H^*_{\mathbb{C}G}$, with $\deg(P_i) = 2i - m$. \end{proof}
	\begin{remark}
		In Proposition~\ref{main thm} and \propref{main thm 2}, if we assume $2m \leq n-k$  then $\phi(u)=0$ whenever $\phi(u) \in H^*_{\mathbb{C}G}$. This is because $H^*_{\mathbb{C}G}$ has no nontrivial relations up to degree $2(n-k)$ and $u^2=0$ implies that $\phi(u)^2=0$ forcing $\phi(u)=0$.
		
	\end{remark}
	A graded endomorphism of $H^*_{\mathbb CG}$ that vanishes on
	$H^2_{\mathbb CG}$ is expected to be trivial, in view of Hoffman’s conjecture \cite{hoffman}. However, unlike the case of the complex Grassmannian, there exist many non-trivial graded endomorphisms of $H^*_\times$ that vanish on $H^2_{\mathbb CG}$. The following proposition provides   such examples when $m$ is even and   $1\le m \le 2k$.

	\begin{proposition}
		For each $i\in I$, choose $P_i \in H^{2i-m}_{\mathbb C G} \subseteq H^*_\times$ and either $Q = au,\, a\in \mathbb Q$, or  $Q\in H^*_{\mathbb CG}\subseteq H^*_{\times}$ with $Q^2=0$  in $ H^*_{\times}$. Then there exist a graded endomorphism $\phi$ on $H^*_\times$ such that 
		\[
		\phi(c_i)=uP_i, \; \forall i\in I, \text{ and  } \quad \phi(u)=Q.
		\]
	\end{proposition}
	\begin{proof}
		Define $\phi$ on $H^*_{\times} = \mathcal{R}/\mathcal{I}$ by $\phi(c_i)=uP_i, \; \forall i\in I, \text{ and } \phi(u)=Q.$ It is sufficient to prove that $\phi$ is well defined, that is, $\mathcal{I}\subseteq \ker (\phi)$. Observe that $u^2 =0$ in $H^*_{\times}$ which implies that \begin{equation}\label{ideal cicj}
			\phi(c_i c_j) = \phi(c_i) \phi(c_j) = uP_i \cdot uP_j = u^2 P_i P_j = 0.
		\end{equation} Using \eqref{ideal cicj} and $\phi(u^2)=Q^2 =0$, we have $\mathcal{I}\subseteq \langle u^2, c_i c_j \,|\, i,j \in I \rangle \subseteq \ker(\phi).$ 
	\end{proof}

	\subsection{} In this subsection, we derive some immediate applications of \propref{main thm}.
	\begin{corollary}
		Let us consider $X = \mathbb{S}^{2m_1} \times \cdots \times \mathbb{S}^{2m_r} \times \mathbb{C}G_{n,k}$ and denote by $u_j$ the generator of $H^{2m_j}(\mathbb{S}^{2m_j}; \mathbb{Q})$ corresponding to the fundamental class of $\mathbb S^{2m_j}$ for all $1\leq j \leq r.$ Define  
		\[
		H^*_{\mathbf{m}, \mathbb{C}G} := H^*(\mathbb{S}^{2m_1} \times \cdots \times \mathbb{S}^{2m_r} \times \mathbb{C}G_{n,k}; \mathbb{Q})
		\;\cong\; H^*_{\mathbb{C}G}[u_1,\ldots,u_r] \big/ \langle u_1^2,\ldots,u_r^2 \rangle,
		\] where $\mathbf{m} = (m_1, \ldots, m_r)$.
		Suppose $\phi: H^*_{\mathbf{m}, \mathbb{C}G} \to H^*_{\mathbf{m}, \mathbb{C}G}$ is a graded endomorphism satisfying $\phi(c_1)=\lambda c_1,\, \lambda \neq 0$. Then $$\phi(c_i) = \begin{cases}
			\lambda^i c_i,   \forall i \in I& \text{ if } k<n-k,\\
			\lambda^i c_i,  \forall i \in I \quad \text{ or } \quad(-\lambda)^i (c^{-1})_i,    \forall i \in I & \text{ if } k= n-k,
		\end{cases}$$
		where $ (c^{-1})_i $ is the $ 2i $-dimensional part of the inverse of $ c = 1 + c_1 + \cdots + c_k $ in $ H^*_{ \mathbb CG} $. 
	\end{corollary}
	\begin{proof}
		The proof of this corollary is similar to the proof of part \textit{2} of \propref{main thm}. Apply induction on $r$ and replace $\mathbb{C}G_{n,k}$ with $\hat{X} := \mathbb{S}^{2m_1}\times\cdots\times \mathbb{S}^{2m_{i-1}}\times \mathbb{S}^{2m_{i+1}}\times\cdots\times \mathbb{S}^{2m_r}\times \mathbb{C}G_{n,k},$ and the sphere $\mathbb{S}^m$ with $\mathbb{S}^{2m_i}$ in \propref{main thm}. Since \begin{equation}\label{Sm as hom}
			\mathbb S^{2m_j}=SO(2m_j+1)/SO(2m_j)
		\end{equation}
		where the orthogonal groups $SO(2m_j+1)$ and $SO(2m_j)$ have the same rank $m_j$. Using \eqref{Sm as hom} and \eqref{cgn as hom}, $\hat{X}$ satisfies the hypothesis of \thmref{Tezuka}. Therefore, every $\mathbb{Q}$-linear derivation of $H^*(\hat{X};\mathbb{Q})$ that decreases the degree by $2m_i$ is trivial.
	\end{proof}
	Let us turn our attention to the generalized Dold spaces $P(m,n,k)$ defined in \subsecref{gds}. The following remark helps us to describe endomorphisms of $H^*(P(m,n,k);\mathbb{Q})$ induced by continuous functions on $P(m,n,k)$. These observations will be used in \secref{section 4}.
	\begin{remark}\label{lift}
		For a continuous map $f$ on $P(m,n,k)$, we have  
		\begin{equation} \label{lift of f}
			f_*\circ \pi_*\big(\pi_1(\mathbb{S}^m\times \mathbb{C}G_{n,k})\big)
			\subseteq \pi_*\big(\pi_1(\mathbb{S}^m\times \mathbb{C}G_{n,k})\big),
		\end{equation}
		where $\pi_1(X)$ denotes the fundamental group of a topological space $X$. Hence, the composite $f\circ \pi$ admits a lift $\tilde f$ on
		$\mathbb{S}^m\times \mathbb{C}G_{n,k}$ for the double covering  
		$\pi:\mathbb{S}^m\times \mathbb{C}G_{n,k}\to P(m,n,k)$.
	\end{remark}
	Using \remref{lift}, we get the following commutative diagram,
	
	\begin{equation}\label{comm diag on H}
		\begin{tikzcd}
			{H^*(P(m,n,k);\mathbb Q)}  & {H^*(\mathbb S^m\times \mathbb CG_{n,k};\mathbb Q)} \\
			{H^*(P(m,n,k);\mathbb Q)} & {H^*(\mathbb S^m\times \mathbb CG_{n,k};\mathbb Q).}
			\arrow["{\pi^*}", from=1-1, to=1-2]
			\arrow["{f^*}"', from=1-1, to=2-1]
			\arrow["{\tilde f^*}", from=1-2, to=2-2]
			\arrow["{\pi^*}", from=2-1, to=2-2]
		\end{tikzcd}
	\end{equation}
	where $\pi^*$ is an injective map. Using \thmref{cohomology of P(m,n,k)} and \eqref{comm diag on H} we obtain the following two corollaries. 
	\begin{corollary}\label{cor3}
		Let $f^*$ be an endomorphism of $H^*(P(m,n,k); \mathbb{Q})$ induced by a continuous function $f$ on $P(m,n,k)$ satisfying $f^*(c_1^2) \ne 0$. Then
		$f^*$ is the restriction of a graded endomorphism $\tilde{f}^*$ on  $H^*_\times$ satisfying $\tilde{f}^*(c_1) = \lambda c_1, \lambda \neq 0$,  to the fixed subring $\mathrm{Fix}(\theta^*)$ of $H^*_\times$ where $\theta = \alpha\times \sigma$.
	\end{corollary}
	\begin{corollary}\label{cor4}
		Let $f^*$ be an endomorphism of $H^*(P(m,n,k); \mathbb{Q})$ induced by a continuous function $f$ on $P(m,n,k)$ satisfying $f^*(c_1^2) =0$ and $n>2$. Then
		$f^*$ is the restriction of a graded endomorphism $\tilde{f}^*$ on  $H^*_\times$ satisfying $\tilde{f}^*(c_1) = au, a \in \mathbb{Q}$,  to the fixed subring $\mathrm{Fix}(\theta^*)$ of $H^*_\times$ where $\theta = \alpha\times \sigma$.
	\end{corollary}
	Using \propref{main thm} in \corref{cor3}, and \propref{main thm 2} in \corref{cor4} along with hypothesis \eqref{Homer}, we can determine $f^*$.
	
	Moreover, there exist graded endomorphisms of $H^*(P(m,n,k))$ that are not induced by any continuous self-map of $P(m,n,k)$, and cannot be realized as restrictions of graded endomorphisms of $H^*_{\times}$. Let us see an example of such graded endomorphism.

	\begin{example}
		If $m$ odd, $n>2$ and $k = 1$, then $P(m,n,1)$ is fibered by the complex projective space  $\mathbb{C} P^{n-1}$ over the real projective space $\mathbb{R} P^m$. In this case,  $H^*_\times\cong \mathbb Q[u,c_1]/\langle u^2, c_1^n\rangle$ and using \eqref{defn of theta*}\ and \thmref{cohomology of P(m,n,k)}, the rational cohomology ring $$H^*(P(m,n,1); \mathbb{Q}) \cong \mathbb{Q}[u, b] / \langle u^2, b^{\lfloor (n+1)/2 \rfloor} \rangle,$$ where $u$ is a generator of $H^m(\mathbb{R}P^m;\mathbb Q)$ and $b$ restricts to $c_1^2\in H^2(\mathbb CP^{n-1};\mathbb Q)$ under the fiber inclusion. 
		
		Consider the endomorphism
		\[
		\phi \colon H^*(P(m,n,1); \mathbb{Q}) \to H^*(P(m,n,1); \mathbb{Q}), \quad \text{defined by }\quad u \mapsto u,  b \mapsto -b.
		\]
		Then $\phi$ is a well-defined graded endomorphism but it cannot be a restriction of a graded endomorphism of $H^*_\times$ because any such map induces $c_1^2 \mapsto \lambda^2 c_1^2$ for some $\lambda \in \mathbb{Q}$, and $\lambda^2 \neq -1$.
	\end{example}

	The following corollary helps us to understand the relationship between the automorphisms of $H^*(P(m,n,k))$ with the automorphisms of $H^*_{\times}$.
	\begin{corollary}\label{automor}
		Let $f^*$ be an automorphism of $H^*(P(m,n,k);\mathbb{Q})$ induced by a continuous function $f$ on $P(m,n,k)$ and assume that $n> 2$. Then $\tilde{f}^*$ is an automorphism of $H^*_{\times}$, where $\tilde{f}$ is as in \remref{lift}. \\ Moreover there exist $\lambda, \mu \in \mathbb{Q}\backslash \{0\}$ such that $\tilde{f}^*(u) = \mu u$ and $\tilde{f}^*(c_i)$ is of the form given in \textit{(2)} of \propref{main thm}.
	\end{corollary}
	\begin{proof}
		Using \remref{lift}, we have $\tilde{f}^*$ is a graded endomorphism of $H^*_{\times}$. When $n>2$, we have $c_1^2\neq 0$ in $\fix (\theta^*)\subseteq H^*_\times$, where Fix$(\theta^*)$ is the fixed subring under $\theta^*$ defined in \eqref{defn of theta*}. Since $f^*$ is an automorphism, we have $f^*(c_1^2) \neq 0$. Using \corref{cor3}, there exist $\lambda \in \mathbb{Q}$ such that $\tilde{f}^*(c_1) = \lambda c_1, \lambda \neq 0$.\\
		By \propref{main thm}, $\tilde{f}^*(c_i)$ is of the form given in \textit{(2)} of \propref{main thm}. Also, $$\tilde{f}^*(u) = \mu u, \, \mu \in \mathbb{Q} \quad \text{ or } \quad \tilde{f}^*(u) = Q$$ where $Q$ is a polynomial of degree $m$ in $H^*_{\mathbb{C}G}$ with $Q^2 =0$. To conclude the result, we need to prove that $\tilde{f}^*(u) = \mu u$ where $\mu \neq 0$. \\
		Suppose that $\tilde{f}^*(u) = Q$, then the image set $\im \tilde{f}^*\subseteq H^*_{\mathbb{C}G}$. Using \corref{cor3}, we get $$\im f^*  \cong \im \tilde{f^*}|_{Fix (\theta^*)}\subseteq H^*_{\mathbb{C}G}.$$ This is a contradiction to the assumption that $f^*$ is an automorphism because using \thmref{cohomology of P(m,n,k)}, either $u$ or $uc_{1}$ (depending on the parity of $m$) is in $\im f^*=\fix (\theta^*) \cong H^*(P(m,n,k);\mathbb{Q})$. Therefore, $\tilde{f}^*(u) = \mu u, \mu \in \mathbb{Q}$ and $\mu \neq 0$ because $f^*$ is an automorphism.
	\end{proof}
	\subsection{} The following theorem shows that if a continuous map on $\mathbb S^m\times \mathbb CG_{n,k}$ stabilizes a copy of Grassmannian, the induced map on cohomology stabilizes the cohomology subalgebra $H^*(\mathbb S^m;\mathbb Q)$.

	\begin{theorem}\label{ind from top}
		Let $f$ be a continuous map on $\mathbb S^m\times \mathbb CG_{n,k}$ such that it stabilizes a copy of Grassmannian $\{{x_0}\}\times \mathbb CG_{n,k}$ for some $x_0\in \mathbb S^m.$ Then the induced endomorphism in cohomology satisfies $f^*(u) = \mu u$ for some $\mu \in \mathbb Z$.
	\end{theorem}


	\begin{proof}
		Let $\mathbb T^m$ be the torus $(\mathbb S^1)^m$ and
		$q:\mathbb T^m\to \mathbb S^m$ be the quotient map that collapses the complement $C$ of an open disk $D\subset \mathbb T^m$ to the point $x_0$ in $\mathbb{S}^m$.  Denote $p_i$ the $i$-th projection map on $\mathbb S^m\times \mathbb CG_{n,k}$ for $i=1,2$ and $s:\mathbb S^m\setminus\{x_0\}\to D$ is the inverse of the restriction $q|_D$.
		Since $f$ stabilizes $\{x_0\}\times \mathbb{C}G_{n,k}$, define continuous maps
		$g:\mathbb{C}G_{n,k}\to \mathbb{C}G_{n,k}$ by $(x_0,g(y)) = f(x_0,y)$ and $\tilde f:\mathbb T^m\times \mathbb{C}G_{n,k}\to \mathbb T^m\times \mathbb{C}G_{n,k}$ by
		\[
		\tilde f(x,y)=
		\begin{cases}
			\big(s\circ p_1\circ f(q(x),y),\,\,p_2\circ f(q(x),y)\big), & x\in D,\\[2pt]
			\big(x,g(y)\big), & x\in C.
		\end{cases}
		\]
		Then it is easy to check that the following diagram commutes:
		\[ \begin{tikzcd}
			{\mathbb T^m\times \mathbb{C}G_{n,k}} & {\mathbb T^m\times \mathbb{C}G_{n,k}} \\
			{\mathbb S^m\times \mathbb{C}G_{n,k}} & {\mathbb S^m\times \mathbb{C}G_{n,k}}
			\arrow["{\tilde f}", from=1-1, to=1-2]
			\arrow["{q\times \mathrm{id}}"', two heads, from=1-1, to=2-1]
			\arrow["{q\times \mathrm{id}}", two heads, from=1-2, to=2-2]
			\arrow["f", from=2-1, to=2-2]
		\end{tikzcd}
		\]
		
		Since, the quotient map $q$ has Brouwer degree 1, the induced map on rational cohomology $q^*: H^*(\mathbb{S}^m; \mathbb{Z}) \rightarrow H^*(\mathbb{T}^m;\mathbb{Z})$ sends $u\mapsto 1\cdot u_1 u_2 \dots u_m$ where $u_i$ denote the one dimensional cohomology class corresponding to the fundamental class of the $i$-th circle factor of $\mathbb T^m$ for $i\in \{1,2,\ldots,m\}$ with appropriate orientation. 
		Since $H^{\mathrm{odd}}(\mathbb{C}G_{n,k};\mathbb Z)=0$, the induced map $\tilde f^*$ sends each $u_i$ to a polynomial $P_i(u_1,\dots,u_m)$.  We slightly abuse notation by using the same symbols for the cohomology classes of $H^*(\mathbb S^m;\mathbb Z)$ and $H^*(\mathbb{C}G_{n,k};\mathbb Z)$ when viewed in $H^*(\mathbb S^m\times \mathbb CG_{n,k};\mathbb Z)$.
		The induced diagram in cohomology implies the following commutative diagram.
		\[
		\begin{tikzcd}
			{\prod_{i=1}^m u_i} & {\prod_{i=1}^m P_i(u_1,\ldots,u_m)} \\
			u & {f^*(u)}
			\arrow["{\tilde f^*}", maps to, from=1-1, to=1-2]
			\arrow["{(q\times \mathrm{id})^*}"', maps to, from=2-1, to=1-1]
			\arrow["{f^*}", maps to, from=2-1, to=2-2]
			\arrow["{(q\times \mathrm{id})^*}"', maps to, from=2-2, to=1-2]
		\end{tikzcd}
		\]
		This implies that $f^*(u)$ does not contain any nonzero element from $H^*(\mathbb{C}G_{n,k};\mathbb Z)$.  
		Thus, $f^*(u)=\mu u$ for some $\mu\in \mathbb Z$.
	\end{proof}

	\section{Coincidence theory of $P(m,n,k)$} \label{section 4}
	In this section, we study the \textit{coincidence theory} of generalized Dold spaces $P(m,n,k)$ defined in \subsecref{gds}. We establish the necessary conditions for a generalized Dold space $P(S,X)$ defined in \eqref{gen dold space} to satisfy the coincidence property. 
	
	\subsection{} Let us recall certain definitions that will be required in the rest of this section.
	
	\begin{definition}
		Let $(X,g)$ be a pair, where $g$ is a continuous map on a topological space $X$. The pair $(X,g)$ is said to have the \textbf{coincidence property} (in short, CP) if, for every continuous map $f : X \to X$, there exists a point $x \in X$ such that $f(x) = g(x)$.
	\end{definition}
	
	If we consider $g$ to be the identity map on $X$, then the notion of coincidence reduces to that of a fixed point, resulting in the following definition.

	\begin{definition}
		A topological space $X$ is said to have \textbf{fixed-point property} (FPP) if every continuous map $f : X \to X$ admits a fixed-point; that is, there exists $x \in X$ such that $f(x) = x$.
	\end{definition}

	The following proposition provides a criteria in terms of the fiber $X$ and the base space $Y := S/\!\!\sim_\alpha$, allowing one to infer the coincidence properties of the total space $P(S,X)$.

	\begin{proposition}\label{necessary condition}
		Let $(P(S,X),g)$ be a pair, where $g$ is a continuous map on the generalized Dold space $P(S,X)$. Then $(P(S,X),g)$ does not have the CP if one of the following hold:
		\begin{enumerate}
			\item The continuous map $g$ is a fiber bundle map and the pair $(Y,p \circ g \circ s)$ does not have the CP, 
			where $Y = S/\!\!\sim_\alpha$ and $s$ denotes a section of the $X$-bundle projection $p$ defined in \eqref{sectio} and \eqref{proj}.
			
			\item  
			There exists a $\sigma$-equivariant map $f$ (i.e. $f\circ\sigma =\sigma\circ f$) on $X$ and a $\alpha \times \sigma$-equivariant map $\tilde g$ on $S\times X$ inducing $g$ such that
			$\mathrm{id}_S \times f$ coincides with neither $\tilde{g}$ nor
			$(\alpha \times \sigma) \circ \tilde{g}$.
		\end{enumerate}

	\end{proposition}
	\begin{proof} \textit{(1)}
		Suppose that the pair $(Y,p \circ g \circ s)$ does not have the CP. 
		Then there exists a continuous map $f : Y \to Y$ such that \begin{equation} \label{pogos}
			f(x) \neq p \circ g \circ s(x), \, \forall x \in Y.
		\end{equation}
		We are given that $g$ is a fiber bundle map, which implies that there exist  $g_1:Y\to Y$, satisfying  $p \circ g = g_1 \circ p.$
		Consider $p\circ g\circ s=g_1\circ p\circ s=g_1 $.
		Thus, $p \circ g = g_1 \circ p$ implies
		\[
		p \circ g(x) = p \circ g \circ s \circ p(x),\, \forall x \in P(S, X).
		\]
		Define the map $\phi := s \circ f \circ p$ on $P(S, X)$. We claim that $\phi(y)\neq g(y), \, \forall y \in P(S,X)$. \\ Suppose there exist $y \in P(S, X)$ such that $\phi(y) = g(y)$, then 
		\[
		p \circ g \circ s(p(y)) = p \circ g(y)
		= p \circ s \circ f \circ p (y) 
		= f(p(y)),
		\]
		which contradicts \eqref{pogos}.

		\textit{(2)} 
		Let $G$ denote the group of deck transformations of the double covering $\pi : S \times X \to P(S, X)$, generated by the free involution $\alpha \times \sigma.$
		The proof then follows from a general observation that if for two $G$-equivariant maps $\tilde\phi, \tilde\psi$ on $S\times X$, the maps $\tilde \phi$ and $t \cdot \tilde\psi$ have no point of coincidence, for any $t \in G$; then the maps they induce on the orbit space $P(S,X)$, namely $\phi, \psi$, are also coincidence-free.
	\end{proof}

	In particular if we take $g$ to be the identity map in \propref{necessary condition}, we recover Proposition~7.2.1 of \cite{mandal}, which proves that if the base $Y$ does not have the FPP, or there exist a $\sigma$-equivarinat map on the fibre $X$ with no fixed point then $P(S,X)$ does not have the FPP. As a consequence, $P(m,n)$ does not have the FPP if either $m$ or $n$ is odd.

	\subsection{} Let us recall a well known result in coincidence theory, the Lefschetz Coincidence Theorem, which will be used to prove results in the rest of this paper.\\
	
	For a closed oriented manifold $M$ of dimension $n$, let $[M]\in H^n(M;\mathbb Q)$ denote a chosen fundamental class. Then we have the Poincar\'e  duality isomorphism $D_M:H^k(M;\mathbb Q)\to H_{n-k}(M;\mathbb Q)$, defined by 
	\begin{equation}\label{poinc dua}
		D_M(\alpha)=[M]\frown \alpha, \, \forall \alpha\in H^k(M;\mathbb Z).
	\end{equation}

	\begin{theorem}[Lefschetz Coincidence Theorem]\label{LCT}
		Let $f,g$ be two continuous maps on a compact, connected, oriented manifold $M$ of  dimension $n$. The Lefschetz coincidence number is defined as 
		$$L(f,g):= \sum_{i=0}^{n} (-1)^i \mathrm{tr}\big(D_M \circ g^* \circ D_M^{-1} \circ f_* \; : \; H_i(M;\mathbb{Q}) \longrightarrow H_i(M;\mathbb{Q})\big).
		$$ If $L(f,g) \neq 0$, then there exists $x \in M$ such that $f(x) = g(x)$.
	\end{theorem}
	When  $g=\operatorname{id}_M$, the theorem reduces to the Lefschetz Fixed-Point Theorem for $M$.\\
	
	To study the coincidence theory of generalized Dold spaces fibred by complex Grassmannians over real projective spaces, it is helpful to first understand the coincidence theory of complex Grassmannians, a topic of independent interest. We now prove the following lemma (cf. Theorem 2, \cite{glover-homer}) to prove \propref{CP of CGnk}.
	\begin{lemma}\label{sum neq 0}
		Let $d_{2i}$ be the $2i$-th Betti number of a complex Grassmannian $\mathbb CG_{n,k}$ with $d=k(n-k)$ even. Then the sum \(\sum _{i=0}^d d_{2i}\lambda^i\neq 0,\, \forall \lambda \in \mathbb Q.\)
	\end{lemma}
	\begin{proof}
		Let us consider the sum $\sum_{i=0}^d d_{2i}\lambda^i$, when $\lambda$ is an integer. 
		Clearly, $$\sum_{i=0}^d d_{2i}\lambda^i \equiv 1 \pmod{\lambda}.$$ 
		Hence, $\sum_{i=0}^d d_{2i}\lambda^i \neq 0$, if $\lambda \neq \pm 1$. When $\lambda = 1$, the sum is also positive and therefore nonzero. It remains to consider the case where $\lambda = -1$.
		Let $\chi(\mathbb{R}G_{n,k})$ denote the Euler-Poincaré characteristic of $\mathbb{R}G_{n,k}$ and be defined by $$\chi(X):=\sum_{i\ge0}\dim H^i(\mathbb{R}G_{n,k};\mathbb{Z}_2)$$ where $\mathbb{R}G_{n,k}$ denotes the Grassmannian of real $k$-planes in $\mathbb{R}^n$.
		Now we observe that 
		$\sum_{i=0}^d d_{2i}(-1)^i = \chi(\mathbb{R}G_{n,k})$ where $d_{2i} = \operatorname{dim} H^{2i}(\mathbb{C}G_{n,k}; \mathbb{Q}) = \dim H^{i}(\mathbb{R}G_{n,k}; \mathbb{Z}_2)$. 
		It is a well known fact that $\chi(\mathbb{R}G_{n,k}) \neq 0$ if $k(n-k)$ is even. 
		
		Let us move to the other case where $\lambda\in \mathbb{Q}\backslash \mathbb{Z}$. Suppose $\sum_{i=0}^{d} d_{2i}\lambda^i =0$ for some $\lambda = \frac{p}{q}$ where $p$ and $q$ are coprime integers. Since $d_0 =d_d =1$, using the rational root theorem $p|1$ and $q|1$. Hence, $\lambda = \pm 1$, which is a contradiction. Therefore, we conclude that $\sum_{i=0}^{d} d_{2i}\lambda^i \neq 0$ for all $\lambda\in \mathbb Q$. 
	\end{proof}

	Denote the $i$-th homology groups 
	$H_i(\mathbb{C}G_{n,k}; \mathbb{Q}), \, H_i(\mathbb{S}^m; \mathbb{Q})$ and 
	$H_i(\mathbb{S}^m \times \mathbb{C}G_{n,k}; \mathbb{Q})$, by $H_i^{\mathbb{C}G}, H_i^{\mathbb{S}}$ and $H_i^{\times}$, respectively. Let $d$ denote the complex dimension of $\mathbb CG_{n,k}$, given by $d = k(n - k)$. Then we have the following proposition.
	\begin{proposition}\label{CP of CGnk}
		Consider a complex Grassmannian $\mathbb{C}G_{n,k}$ such that the hypothesis \eqref{Homer} is satisfied and $k(n-k)$ is even. Let $g $ be a continuous map on $\mathbb{C}G_{n,k}$ with nonzero Brouwer degree. Then the pair $(\mathbb{C}G_{n,k}, g)$ has the coincidence property.
	\end{proposition}

	\begin{proof}
		Self-maps with nonzero Brouwer degree induces automorphisms in the rational cohomology algebra. Using Theorem \ref{hom and hof} part \textit{(i)}, there exist a nonzero rational $\lambda$ such that $g^*(c_i)=\lambda^ic_i, \forall i\in I.$ Let $f$ be a continuous map on $\mathbb CG_{n,k}$ and using \thmref{hom and hof} part \textit{(i)}, there exists $\mu \in \mathbb Q$ such that
		\[
		f^*(c_i)=\mu^ic_i, \forall i\in I.
		\]
		Then by the Universal Coefficient Theorem, $ \hom_{\mathbb{Q}} (H_i^{\mathbb CG};\mathbb Q) \cong H^i_{\mathbb CG}$ non-canonically which implies that 
		\begin{align*}
			\varphi \circ f_* &= f^*(\varphi) , \, \forall
			\varphi \in \hom _{\mathbb Q}\big(H_{2i}^{\mathbb CG}, \mathbb Q\big)\cong H^{2i}_{\mathbb CG}.\\
			\varphi(f_*(x))&=(f^*(\varphi))(x)= \mu^i\varphi(x)=\varphi(\mu^ix),\; \forall x\in H_{2i}^{\mathbb CG}.
		\end{align*}
		The last equation implies that $f_*(x)=\mu^ix,\, \forall x\in  H_{2i}^{\mathbb CG}.$
		Now observe that $D\circ g^*\circ D^{-1}\circ f_*: H_{2i}^{\mathbb CG}\to H_{2i}^{\mathbb CG}$ is given by 
		\[
		D\circ g^*\circ D^{-1}\circ f_*(x)=D\circ g^*\circ D^{-1}(\mu^ix)=\mu^iD\circ g^*(D^{-1}x)=\mu^iD(\lambda^{d-i}D^{-1}x)=\mu^i\lambda^{d-i}x.
		\]
		Thus for $x\in H_{2i}^{\mathbb{C}G}$, the Lefschetz coincidence number is given by
		\[
		\begin{array}{ll}
			L(f,g) &= \sum_{i=0}^d (-1)^{2i}
			\mathrm{tr}(D\circ g^*\circ D^{-1}\circ f_*(x)) \\[6pt]
			&=  \sum_{i=0}^d d_{2i}\mu^i\lambda^{d-i}  \\[6pt]
			&= \lambda^d \sum_{i=0}^d d_{2i}(\mu/\lambda)^i \neq 0 \quad (\because \lambda\neq 0) 
		\end{array}
		\]
		where $d_{2i}$ denotes $\dim_{\mathbb Q}H^{2i}_{\mathbb CG}$ and the last equation holds by using \lemref{sum neq 0}. Therefore, using \thmref{LCT} the pair $(\mathbb{C}G_{n,k},g)$ has the coincidence property.
	\end{proof}

	\subsection{} Denote by $H_*^\times = \bigoplus_{i\geq 0} H_i^{\times},\, H_*^{\mathbb{C}G} = \bigoplus_{i\geq 0} H_i^{\mathbb{C}G},\, H_*^{\mathbb{S}} = \bigoplus_{i\geq 0} H_i^{\mathbb{S}}$ and $\vartheta$ the fundamental class $[\mathbb S^m]\in H_m^{\mathbb S}$. Let $\{v_q\}$ be a homogeneous basis of $H_*^{\mathbb{C}G}$, and let $\{\delta_{v_q}\}$ denote the corresponding dual basis of  
	$\operatorname{Hom}(H_*^{\mathbb{C}G}, \mathbb Q) \cong H^*_{\mathbb{C}G}$, such that  
	$ \delta_{v_q} (v_p) = \delta_{qp}$ where $\delta_{qp}$ is the Kronecker delta function. Without loss of generality, assume that $1=v_0 \in \{v_i\}$ represents the generator of $H_0^{\mathbb{C}G} \cong \mathbb Q$.
	
	Over $\mathbb Q$, the K\"unneth Theorem yields the following decompositions 
	\begin{equation}\label{kunneth}
		H_i^\times \cong H_i^{\mathbb{C}G} \oplus (\vartheta \otimes H_{i-m}^{\mathbb{C}G}), 
		\qquad 
		H^i_\times \cong H^i_{\mathbb{C}G} \oplus uH^{i-m}_{\mathbb{C}G},
	\end{equation}
	where  $u \in H^m_{\times} \cong \operatorname{Hom}(H_m^{\times}, \mathbb Q)$ corresponds to the element $\delta_{\vartheta\otimes 1}$.

	Using \eqref{kunneth}, we can extend the chosen basis $\{v_q\}$ of $H_*^{\mathbb{C}G}$ to $	\{v_q\} \cup \{\vartheta \otimes v_q\}$ of $H_*^\times$ such that the corresponding dual basis can also be extended from $\{\delta_{v_q}\}$ of $ \hom (H_*^{\mathbb{C}G};\mathbb{Q})$ to $\{\delta_{v_q}\} \cup \{\delta_{\vartheta \otimes v_q}\} $ of $\hom (H_*^{\times};\mathbb{Q})$ satisfying:
	\begin{equation}\label{Kronecker relations}
		\delta_{v_q}( v_p) = \delta_{qp}, \quad
		\delta_{v_q}( \vartheta \otimes v_p)= 0, \quad
		\delta_{\vartheta\otimes v_q}(v_p) = 0, \quad
		\delta_{\vartheta\otimes v_q}(\vartheta \otimes v_p)= \delta_{qp}.
	\end{equation}
	Let $f$ be a continuous function on $P(m,n,k)$. Using \remref{lift} and the Universal Coefficient Theorem, there exist a lift $\tilde{f}$ on $\mathbb{S}^m\times \mathbb{C}G_{n,k}$ satisfying \begin{equation}\label{comm with phi}
		\varphi \circ \tilde{f}_* = \tilde{f}^*(\varphi) , \, \forall
		\varphi \in \hom _{\mathbb Q}\big(H_{2i}^{\mathbb CG}, \mathbb Q\big)\cong H^{2i}_{\mathbb CG}.
	\end{equation}
	
	Poincar\'e duality on $\mathbb S^{m}\times \mathbb C G_{n,k}$ can be described in terms of the duality on the Grassmannian factor. Let 
	$D_{\mathbb C G}\colon H^{i}_{\mathbb C G}\to H_{2d-i}^{\mathbb C G}$ 
	be the Poincar\'e duality isomorphism defined in \eqref{poinc dua} for $\mathbb C G_{n,k}$, where $d=k(n-k)$.  
	The Poincar\'e duality isomorphism on the product
	is then determined on the basis elements by  
	\begin{equation}
		D\colon H^{j}_{\times}\to H_{m+2d-j}^{\times}, \quad \delta_{v_i}\mapsto \vartheta\otimes D_{\mathbb C G}(\delta_{v_i})\quad\text{and}\quad \delta_{\vartheta\otimes v_i} \mapsto D_{\mathbb C G}(\delta_{v_i}).
	\end{equation}
	We are now ready to establish the following lemmas, which will be useful in the sequel.
	\begin{lemma}\label{image of homf}
		Let $f$ be a continuous function on $P(m,n,k)$ and $\tilde{f}$ be the lift defined in \remref{lift} such that $\tilde{f}^*(c_1)\neq au,\, a\in \mathbb{Q}$ and $k<n-k$. Then there exist $\lambda \in \mathbb{Q}\backslash \{0\}$ and $\mu \in \mathbb{Q}$ such that the induced map $\tilde{f}_*$ on $H_*^{\times}$ is of the following form.
		\begin{enumerate}
			\item Either $\tilde{f}_*(\vartheta\otimes x) = \mu \lambda^i (\vartheta \otimes x),\, \forall x\in H_{2i}^{\mathbb{C}G}$ or $\tilde{f}_*(\vartheta\otimes x) \in H_*^{\mathbb{CG}},\, \forall x \in H_*^{\mathbb{C}G}$.
			\item $\tilde{f}_*(x) = \lambda^i x + \vartheta\otimes y,$ for some $y \in H_{2i-m}^{\mathbb{C}G}, \, \forall x \in H_{2i}^{\mathbb{C}G}.$
		\end{enumerate}
		Moreover, $y=0$ in \textit{(2)} if $\tilde{f}_*(\vartheta\otimes x) = \mu \lambda^i (\vartheta \otimes x),\, \forall x\in H_{2i}^{\mathbb{C}G}$.
	\end{lemma}
	\begin{proof}
		Using \propref{main thm}, there exist $\lambda \in \mathbb{Q}\backslash \{0\}$ such that $\tilde{f}^*(c_i) = \lambda^i c_i, \forall i \in I$ and either $\tilde{f}^*(u) = \mu u,\, \mu \in \mathbb{Q}$ or $\tilde{f}^*(u)\in H^*_{\mathbb{C}G}$. It is sufficient to prove the result for the chosen basis $\{v_q\}\cup\{\vartheta\otimes v_q\}$ of $H_{*}^{\times}$. 
		
		Let us consider the first case where $\tilde{f}^*(u) = \mu u$. Using $H^*_{\mathbb{C}G}\cong \hom(H_*^{\mathbb{C}G},\mathbb{Q})$, we have \begin{equation}\label{fstarco}
			\tilde{f}^*(\delta_{v_p}) = \lambda^i \delta_{v_{p}}, \, \forall v_{p}\in H_{2i}^{\mathbb{C}G}, \quad \tilde{f}^*(\delta_{\vartheta\otimes v_{p}})  = \mu \lambda^i (\delta_{\vartheta \otimes v_{p}}), \, \forall v_{p}\in H_{2i}^{\mathbb{C}G}.
		\end{equation}
		If $m$ is odd, then the coefficient of any basis element $v_p \in H_*^{\mathbb{C}G}$ in $\tilde{f}_*(\vartheta \otimes v_q)$ and $\vartheta \otimes v_p$ in $\tilde{f}_*(v_q)$ is zero because $\tilde{f}_*$ is a graded map. Let us consider the case where $m=2s$.
		By \eqref{comm with phi} and \eqref{fstarco}, the coefficient of a basis element $v_p\in H_{2i+m}^{\mathbb{C}G}$ in $\tilde{f}_*(\vartheta \otimes v_q)$ written as a $\mathbb{Q}$-linear combination of the basis elements from $\{v_q\}\cup\{\vartheta\otimes v_q\}$ is the following:
		$$	 \delta_{v_p}\circ \tilde f_*(\vartheta\otimes v_q)
		=  \tilde f^* (\delta_{v_p}) (\vartheta\otimes v_q)
		=  \lambda^{i+s} \delta_{v_p}(\vartheta\otimes v_q)
		= 0,\, \forall v_q \in H_{2i}^{\mathbb{C}G}$$
		and the coefficient of a basis element $\vartheta \otimes v_p\in \vartheta \otimes H_{2i}^{\mathbb{C}G}$ in $\tilde{f}_*(\vartheta \otimes v_q)$ is
		$$ \delta_{\vartheta\otimes v_p}\circ \tilde f_*(\vartheta\otimes v_q)
		=  \tilde f^*(\delta_{\vartheta\otimes v_p})(\vartheta\otimes v_q)
		=  \mu \lambda^{i}\;\delta_{\vartheta \otimes v_p}(\vartheta\otimes v_q)= \mu\lambda^{i}\delta_{pq},\, \forall v_q \in H^{\mathbb{C}G}_{2i}.$$
		This implies that $$\tilde{f}_*(\vartheta \otimes v_q) = \mu \lambda^i (\vartheta \otimes v_q), \, \forall v_q \in H_{2i}^{\mathbb{C}G}.$$ 
		Using similar calculations given above, it is easy to show that $$\delta_{v_p}\circ \tilde{f}_*(v_q) = \lambda^i \delta_{pq},\, \forall v_q \in H_{2i}^{\mathbb{C}G},\quad \delta_{\vartheta \otimes v_p}\circ \tilde{f}_*(v_q)=0,\, \forall v_q\in H_{2i}^{\mathbb{C}G}.$$ Therefore, $\tilde{f}_*(v_q) = \lambda^i v_q,\, \forall v_q \in H_{2i}^{\mathbb{C}G}$.\\
		
		If $\tilde{f}^*(u)\in H^*_{\mathbb{C}G}$. Again using $H^*_{\mathbb{C}G}\cong \hom(H_*^{\mathbb{C}G},\mathbb{Q})$, we have \begin{equation}\label{second u}
			\tilde{f}^*(\delta_{v_p}) = \lambda^i \delta_{v_{p}}, \, \forall v_{p}\in H_{2i}^{\mathbb{C}G}, \quad \tilde{f}^*(\delta_{\vartheta\otimes v_{p}})  \in H^*_{\mathbb{C}G},\, \forall  v_{p}\in H_{2i}^{\mathbb{C}G}.
		\end{equation}
		By \eqref{comm with phi} and \eqref{second u}, we get $\delta_{v_p}\circ \tilde{f}_*(v_q) = \lambda^i \delta_{pq},\, \forall v_q \in H_{2i}^{\mathbb{C}G},$ which implies that $\tilde{f}_*(x) = \lambda^i x + \vartheta\otimes y,$ for some $y \in H_{2i-m}^{\mathbb{C}G}, \, \forall x \in H_{2i}^{\mathbb{C}G}.$ Note that $\tilde f^*(\delta_{\vartheta\otimes v_p})\in H_{\mathbb CG}^*$ and equal to some $\sum a_j\delta_{v_j}$. Then 
		\begin{equation}\label{computation}
			\delta_{\vartheta\otimes v_p}\circ \tilde f_*(\vartheta\otimes v_q)= \tilde f^*(\delta_{\vartheta\otimes v_p})(\vartheta\otimes v_q)=\sum a_j\delta_{v_j}(\vartheta\otimes v_q)=0.
		\end{equation}
		Hence, $\tilde f_*(\vartheta \otimes v_q)\in H_*^{\mathbb CG}$ for all $\vartheta \otimes v_q\in\vartheta\otimes H_*^{\mathbb CG} $.
	\end{proof}
	\begin{lemma}\label{image of homf under hom}
		Assume that the hypothesis \eqref{Homer} is satisfied. 	Let $f$ be a continuous function on $P(m,n,k)$ and $\tilde{f}$ be the lift defined in \remref{lift} such that $\tilde{f}^*(c_1)= au,\, a\in \mathbb{Q}$. Then the induced map $\tilde{f}_*$ on $H_{*}^{\times}$ is of the following form.
		\begin{enumerate}
			\item $\tilde f_* (x)\in \vartheta\otimes H_{2i-m}^{\mathbb CG},\, \forall x\in H_{2i}^{\mathbb CG},\, \forall i>0$.
			\item $\tilde{f}_*(\vartheta\otimes 1) = \mu (\vartheta\otimes 1 )+y, \, y\in H_{m}^{\mathbb{C}G}, \quad\\ \tilde f_*(\vartheta \otimes x)\in H_{2i+m}^{\mathbb{C}G},\, \forall x \in H_{2i}^{\mathbb{C}G}, i>0\;$ if $\tilde{f}^*(u) = \mu u,\,  \mu \in \mathbb{Q}$ 
		\end{enumerate}
	\end{lemma}
	\begin{proof}
		Using \propref{main thm 2}, we have $\tilde{f}^*(c_i) = u P_i,$ for some $P_i\in H^*_{\mathbb{C}G}$ and either $\tilde{f}^*(u) = \mu u,\, \mu \in \mathbb{Q}$ or $\tilde{f}^*(u)\in H^*_{\mathbb{C}G}$. \\
		Let us consider the first case where $\tilde{f}^*(u) = \mu u$. Using $H^*_{\mathbb{C}G}\cong \hom(H_*^{\mathbb{C}G},\mathbb{Q})$, we have for $i>0$ 
		\begin{equation}\label{fstarcom}
			\tilde{f}^*(\delta_{v_p}) = \sum a_{jp}\delta_{\vartheta \otimes v_j} , \, \forall v_{p}\in H_{2i}^{\mathbb{C}G}, \quad \tilde{f}^*(\delta_{\vartheta\otimes 1})  =\mu \delta_{\vartheta \otimes 1}, \quad \tilde{f}^*(\delta_{\vartheta\otimes v_{p}})  = 0, \, \forall v_{p}\in H_{2i}^{\mathbb{C}G}.
		\end{equation}
		Using \eqref{comm with phi}, \eqref{fstarcom} and similar calculations given in the proof of \lemref{image of homf}, we have for $ v_p\neq 1$
		$$\delta_{v_p}\circ \tilde{f}_*(v_q) = 0,\, \forall v_q \in H_{2i}^{\mathbb{C}G},\quad \delta_{\vartheta\otimes v_p}\circ \tilde{f}_*(\vartheta\otimes v_q) =0,\, \forall v_q \in H_{2i}^{*}$$ that concludes the result.
		
		When $\tilde{f}^*(u)\in H^*_{\mathbb{C}G}$ then also we have $\tilde{f}^*(\delta_{v_p}) = \sum a_{jp}\delta_{\vartheta \otimes v_j} , \, \forall v_{p}\in H_{2i}^{\mathbb{C}G},\, \forall i>0$ which implies that $\delta_{v_p}\circ \tilde{f}_*(v_q) = 0,\, \forall v_q \in H_{2i}^{\mathbb{C}G},\,\forall i>0.$   		
	\end{proof}

\subsection{} The following theorems provide a criteria for the existence of coincidence points between a pair of continuous functions on $P(m,n,k)$.

\begin{theorem}\label{coincidence thm}
	Let $P(m,n,k)$ be a generalized Dold manifold with $k<n-k$ and $k(n-k)$ even. Let $f$ and $g$ be two continuous maps on $P(m,n,k)$ and $\tilde f, \tilde g$ be their lifts as defined in  Remark \ref{lift} such that
	\begin{enumerate}	
		\item $g^*$ is an automorphism of $H^*(P(m,n,k);\mathbb Q)$. 
		\item $\tilde{f}^*(c_1) \neq au,\, a \in \mathbb{Q}$.
		\item $\deg(p\circ g \circ s)\neq -\deg (p\circ f\circ s)$ if $m$ is odd.
	\end{enumerate}
	where $s$ denotes a section of the $\mathbb CG_{n,k}$-bundle projection $p: P(m,n,k)\to \mathbb RP^m$.	Then, there is a point of coincidence of $f$ and  $g$.
\end{theorem}
\begin{proof}
	Using \corref{automor}, we have $\tilde g^*$ is an automorphism on $H^*_{\times}$ given by
	$\tilde g^*(c_i) = \lambda_1^i c_i$, and $\tilde g^*(u) = \mu_1 u$ for some $\lambda_1, \mu_1 \in \mathbb{Q}\backslash \{0\}$ if $k<n-k$.
	
	Using \lemref{image of homf}, there exist $\lambda \in \mathbb{Q}\backslash \{0\}$ and $\mu \in \mathbb{Q}$ such that $\tilde{f}_*$ is of the following form, 
	\begin{equation}\label{flower star}
		\begin{split}
			\tilde f_*(x)=\lambda^ix+\vartheta\otimes y, \text{ for some }y\in H_{2i-m}^{\mathbb CG}, \, \forall x \in H_{2i}^{\mathbb{C}G}\\
			\tilde f_*(\vartheta\otimes x)=\mu\lambda^i (\vartheta\otimes x) , \text{ or } \tilde{f}_*(\vartheta \otimes x) = z, \text{ for some }z\in H_{2i+m}^{\mathbb CG},\, \forall x \in H_{2i}^{\mathbb{C}G}
		\end{split}
	\end{equation}
	To prove that $f$ has a point of coincidence with $g$, it is sufficient to prove that either $\tilde{f}$ or the composition $\theta \circ \tilde{f} $ has a point of coincidence with $g$ where $\theta = \alpha \times \sigma$ defined in \secref{gds}. By \thmref{LCT}, we need to compute $L(\tilde f, \tilde g)$ and $L(\theta \circ \tilde f, \tilde g)$.
	
	For $x\in H_{2i}^{\mathbb{C}G}$, we have 
	\begin{equation*}\label{D cal}
		\begin{split}
			D\tilde g^* D^{-1} \tilde f_*(x)=\mu_1 \lambda^i\lambda_1^{d-i} x + \vartheta\otimes y'\text{ for some }y' \in H^{\mathbb CG}_{2i-m}\\
			D\tilde g^* D^{-1} \tilde f_*(\vartheta\otimes x)= \mu \lambda^i\lambda_1^{d-i}(\vartheta\otimes x)+z' \text{ for some }z'\in H_{2i+m}^{\mathbb CG}.
		\end{split}
	\end{equation*}
	where $z^{'} =0$ or $\mu =0$ depending on the image of $\tilde{f}_*(\vartheta \otimes x)$.
	Recall that  $d_{2i}$ denote the dimension $\dim H^{2i}_{\mathbb CG}$. The Lefschetz number  $L(\tilde{f},\tilde g)$ is
	\begin{equation*}\label{L(f,g)}
		L(\tilde f,\tilde g) =(\mu_1 + \mu) \sum_{i=0}^{k(n-k)} d_{2i} \lambda^i\lambda_1^{d-i}.
	\end{equation*}
	Using the \lemref{sum neq 0} and the fact that $\lambda_1\neq 0$, the sum
	\[
	\sum_{i=0}^{k(n-k)} d_{2i} \lambda^i\lambda_1^{d-i}=\lambda_1^d\sum_{i=0}^{k(n-k)} d_{2i} (\lambda/\lambda_1)^i\neq 0,
	\]
	Since $\tilde f\circ\theta=\theta\circ \tilde f$, it follows that $$(\theta\circ \tilde  f )^*(c_i)= (-1)^i \tilde{f}^*(c_i), \forall i \in I, \quad (\theta \circ\tilde  f)^*(u)=\begin{cases}
		-\tilde  f^*(u), \text{ if } m \text{ is even,}\\
		\tilde  f^*(u), \text{ if } m \text{ is odd}.
	\end{cases}$$ 
	If $m$ is even, then  
	\begin{equation*}\label{D with theta}
		\begin{split}
			D\tilde g^* D^{-1} (\theta\circ \tilde f)_*(x)=\mu_1(- \lambda)^i\lambda_1^{d-i} x + \vartheta\otimes y''\text{ for some }y'' \in H^{\mathbb CG}_{2i-m} \\
			D\tilde g^* D^{-1} (\theta\circ \tilde f)_*(\vartheta\otimes x)= -\mu (-\lambda)^i\lambda_1^{d-i}\vartheta\otimes x+z'' \text{ for some }z''\in H_{2i+m}^{\mathbb CG}.
		\end{split}
	\end{equation*}
	Thus,  the Lefschetz number is
	\begin{equation*}\label{L(theta f,g)}
		L(\theta \circ\tilde f,\tilde g) =(\mu_1 - \mu)\sum_{i=0}^{k(n-k)} d_{2i} (-\lambda)^i\lambda_1^{d-i}.
	\end{equation*}
	Also, using  $\mu_1\neq 0$ and \lemref{sum neq 0}  it follows that that either $L(\tilde f, \tilde g)$ or $L(\theta\circ \tilde f,\tilde g)$ is nonzero. 
	
	If $m$ is odd, $	L(\theta \circ\tilde f,\tilde g) =(\mu_1 + \mu)\sum_{i=0}^{k(n-k)} d_{2i} (-\lambda)^i\lambda_1^{d-i}.$ Using  \lemref{sum neq 0} and $\deg(p\circ g \circ s)\neq -\deg (p\circ f\circ s)$ that is $\mu_1 \neq -\mu$, we have both $L(\tilde{f},\tilde{g})$ and $L(\theta \circ\tilde f,\tilde g)$ are nonzero. 
	This ensures that there exist a point of conincidence between $f $ and $g$.
\end{proof}
\begin{theorem}\label{coincidence thm under hom}
	Let $P(m,n,k)$ be a generalized Dold manifold with $k(n-k)$ even and assume that the hypothesis \eqref{Homer} is satisfied. Let $g$ and $f$ are two continuous maps on $P(m,n,k)$ and $\tilde g, \tilde f$ be their lifts as defined in  Remark \ref{lift} such that
	\begin{enumerate}	
		\item $g^*$ is an automorphism of $H^*(P(m,n,k);\mathbb Q)$. 
		\item $\tilde{f}^*(u) = \mu u,\, \mu \in \mathbb{Q}$ if $\tilde{f}^*(H^*_{\mathbb CG}) \nsubseteq H^*_{\mathbb CG}$ and $m$ is even.
		\item $\deg(p\circ g \circ s)\neq -\deg (p\circ f\circ s)$ if $m$ is odd.
	\end{enumerate}
	where $s$ denotes a section of the $\mathbb CG_{n,k}$-bundle projection $p: P(m,n,k)\to \mathbb RP^m$. Then, there is a point of coincidence of $f$ and  $g$.
\end{theorem}
\begin{proof}
	If $\tilde{f}^*(c_1) \neq au, \, a\in \mathbb{Q}$ then we have the result by \thmref{coincidence thm}.\\ Let us consider the other case when $\tilde{f}^*(c_1) = au, \, a\in \mathbb{Q}$, using \propref{main thm 2} we have $\tilde{f}^*(c_i) = uP_i, \text{ for some } P_i\in H^{2i-m}_{\mathbb{C}G}.$ 
	
	If $P_i \neq 0$ for some $i$ in $I$ then $\tilde{f}^*(H^*_{\mathbb CG}) \nsubseteq H^*_{\mathbb CG}$. Since $\tilde{f}^*$ is graded and by \textit{(2)} we have $\tilde{f}^*(u) = \mu u,\, \mu \in \mathbb{Q}$. Using \lemref{image of homf under hom}, $\tilde{f}_*$ is of the following form, \begin{equation}\label{uin CG}
		\begin{split}
			\tilde f_*(x)= \vartheta\otimes y, \text{ for some }y\in H_{2i-m}^{\mathbb CG}, \, \forall x \in H_{2i}^{\mathbb{C}G}, \, i>0\\
			\tilde f_*(\vartheta\otimes x)=\mu(\vartheta\otimes x) +z, \text{ for some }z\in H_{2i+m}^{\mathbb CG},\, \forall x \in H_{2i}^{\mathbb{C}G}
		\end{split}
	\end{equation}
	where $\mu =0$ if $i>0$. By \corref{automor}, we have $\tilde g^*$ is an automorphism on $H^*_{\times}$ given by
	$\tilde g^*(c_i) = \lambda_1^i c_i$, and $\tilde g^*(u) = \mu_1 u$ for some $\lambda_1, \mu_1 \in \mathbb{Q}\backslash \{0\}$. Using \thmref{LCT} and the similar calculations as done in the proof of \thmref{coincidence thm}, we get $$L(\tilde f,\tilde g) =(\mu_1 + \mu) d_0 \lambda_1^d, \quad L(\theta \circ\tilde f,\tilde g) =\begin{cases}
		(\mu_1 - \mu)d_{0} \lambda_1^{d},  \text{ if } m \text{ is even,}\\
		(\mu_1 +\mu)d_{0} \lambda_1^{d},  \text{ if } m \text{ is odd.}
	\end{cases}$$ Using $\lambda_1 \neq 0$ and $\mu_1 \neq 0$, either $L(\tilde f,\tilde g) $ or $ L(\theta \circ\tilde f,\tilde g)$ is non zero if $m$ is even. Using $\deg(p\circ g \circ s)\neq -\deg (p\circ f\circ s)$ i.e. $\mu_1 \neq -\mu$ we have $L(\tilde f, \tilde g) = L(\theta \circ \tilde f, \tilde g) \neq 0$.  Hence, we get the result.
	
	Let us consider the case when $P_i =0,\, \forall i \in I$, if $\tilde{f}^*(u) = \mu u, \mu \in \mathbb{Q}$ then the proof remains exactly the same as given above. We need to focus on the case when $\tilde{f}^*(u)\in H^*_{\mathbb{C}G}$. Using \lemref{image of homf under hom} and \eqref{computation}, we have $$\tilde{f}_*(x) = \vartheta \otimes y, \text{ for some }y\in H_{2i-m}^{\mathbb CG}, \, \forall x \in H_{2i}^{\mathbb{C}G}, \, i>0, \quad \tilde{f}_*(\vartheta \otimes x) \in H_*^{\mathbb{C}G}, \forall x\in  H_*^{\mathbb{C}G}.$$ This is exactly the same if we take $\mu =0$ in \eqref{uin CG}. The rest of the calculations also remains the same and we get the result.
\end{proof}

\begin{remark}
	There are many situations when the map $f$ satisfies the required hypothesis \textit{(2)} considered in \thmref{coincidence thm} or \thmref{coincidence thm under hom}. Some of them are as follows: 
	\begin{enumerate}
		\item The lift $\tilde f$ stabilizes a copy of Grassmannian, i.e., $\tilde f(\{x_0\}\times\mathbb CG_{n,k})\subseteq \{x_0\}\times\mathbb CG_{n,k}$ for some $x_0\in \mathbb S^m$. 
		\item  The map $p_1\circ \tilde f^*\circ i_1: H^*_{\mathbb CG}\to H^*_{\mathbb C G}$ is an automorphism, equivalently, $f^*(c_1^2)=\lambda^2c_1^2,\, \lambda\in \mathbb Q\backslash \{0\},$ where $p_1$ and $i_1$ are defined in \eqref{comm diagram}.
		\item The map $p_2\circ \tilde f\circ i_1:\mathbb S^m\to \mathbb CG_{n,k}$ is rationally null homotopic, where $p_2$ is the projection onto the second summand and $i_1$ is the inclusion into the first summand. 
	\end{enumerate} 
\end{remark}

Under the assumption $m>2k$, any continuous map $f$ on the generalized Dold space $P(m,n,k)$, the lift $\tilde f$ (from Remark~\ref{lift}) satisfies $\tilde f^{*}(c_i)=\lambda^ic_i$ for all $i\in I$. Hence condition~\textit{(2)} of Theorem~\ref{coincidence thm under hom} may be omitted, and one obtains the following consequence.
\begin{corollary}
	Let $P(m,n,k)$ be a generalized Dold space with $m$ and $k(n-k)$ both even. Assume $m>2k$, and the hypothesis \eqref{Homer} is satisfied.
	Then, for any continuous function $g$ on $P(m,n,k)$ that induces an automorphism on $H^*(P(m,n,k);\mathbb{Q})$, the pair $(P(m,n,k),g)$ has the coincidence property. \\ In particular, for $g=\mathrm{id}$, the space $P(m,n,k)$ has the fixed-point property.
\end{corollary}

In \thmref{coincidence thm under hom}, the first assumption that $g^*$ is an automorphism of $H^*(P(m,n,k); \mathbb{Q})$ can be relaxed by assuming $\mu$ is nonzero, which leads to the following proposition. 
\begin{proposition}
	Let $P(m,n,k)$ be a generalized Dold manifold with $k(n-k)$ even and assume that the hypothesis \eqref{Homer} is satisfied. Let $g$ and $f$ are two continuous maps on $P(m,n,k)$ and $\tilde g, \tilde f$ be their lifts as defined in  Remark \ref{lift} such that
	\begin{enumerate}
		\item $\tilde g^*(H^*_{\mathbb CG})= H^*_{\mathbb CG}$.
		\item $\tilde f^*(u)=\mu u,\, \mu\in \mathbb Q\backslash \{0\}$
	\end{enumerate}
	Then, there is a point of coincidence of $f$ and  $g$.
\end{proposition}
The proof of the above proposition is similar to the proof of \thmref{coincidence thm under hom}. Therefore, we omit the details.

\section*{Acknowledgements}
Part of this work was carried out while the first author was a postdoctoral fellow at IISER Berhampur, which the author gratefully acknowledges.


\end{document}